\documentclass[10pt]{amsart}
\usepackage{amsmath}
\usepackage{amsfonts}
\usepackage{amssymb}
\usepackage{graphicx}
\usepackage[active]{srcltx}
\usepackage[all]{xy}
\usepackage{graphics}
\usepackage{color}
\usepackage[utf8]{inputenc}
\usepackage{mathtools}
\usepackage{tikz}
\usepackage{dsfont}
\usepackage{tikz-cd}
\usetikzlibrary{decorations.pathmorphing}
\usepackage{float}

\theoremstyle{plain}

\swapnumbers
\newtheorem{prop}{Proposition}[section]
\newtheorem{thm}[prop]{Theorem}
\newtheorem{cor}[prop]{Corollary}

\theoremstyle{definition}
\newtheorem{point}[prop]{}
\newtheorem{defn}[prop]{Definition}

\begin{document}

\title{Homotopical decompositions of simplicial  and    Vietoris  Rips  complexes
}


\author{Wojciech Chach\'olski}
\thanks{Mathematics, KTH, S-10044
Stockholm, Sweden\email{wojtek@kth.se} \email{alvinj@kth.se}\email{scola@kth.se}  \email{tombari@kth.se} }
\author{ Alvin Jin }
\author{Martina Scolamiero}
\author{Francesca Tombari}

\maketitle

\begin{abstract}
Motivated by applications in Topological Data Analysis,  we consider decompositions of a simplicial complex induced by a cover of its vertices. We study how the homotopy type of such decompositions approximates the homotopy of the simplicial complex itself. The difference between the simplicial complex and such an approximation is quantitatively measured by means of the so called obstruction complexes. Our general machinery is then specialized to clique complexes, Vietoris-Rips complexes and Vietoris-Rips complexes of metric gluings.
For the latter we give metric conditions which allow to recover the first and zero-th homology of the gluing from the respective homologies of the components.

\end{abstract}

\section{Introduction}
\label{intro}
Homology is an example of an invariant  that is both calculable and geometrically informative. These two features  are key reasons  why invariants derived from homology   are fundamental  in Algebraic Topology in general and in Topological Data Analysis (TDA)~\cite{MR2476414}
in particular.  Calculability  is a consequence of the fact that homology converts homotopy push-outs into   Mayer-Vietoris exact sequences.  Decomposing a space into a homotopy push-out enables us  to extract the homologies of the decomposed space (global information) from the homologies of the spaces in the push-out (local information).   

Ability of extracting global information from  local is important. What is meant by local information  however depends on the 
input and the description of considered spaces. For example what is often understood as local information in TDA differs from
the local information described above (push-out decomposition). In TDA the input is typically a finite metric space. This information is then converted into spacial  information and in this article we focus on the so called Vietoris-Rips construction~\cite{MR1368659} for that purpose.   Homologies extracted from this space give rise to  invariants of the metric space used in TDA such as persistent homologies~\cite{MR1861130,Delf-Edel,MR2405684,MR1434543,Frosini-Landi}, bar-codes~\cite{Scopigno04persistencebarcodes}, stable ranks~\cite{MR4057607,MR3735858}, or persistent landscapes~\cite{MR3317230}.  This conversion process, from metric  into spacial information,  does not in general transform the gluing  of metric spaces~\cite{MR1400226} into  homotopy push-outs and homotopy colimits of simplicial complexes.  The aim of this paper is to understand how close such data driven decompositions are to decompositions into homotopy push-outs. Our work was inspired by~\cite{MR3824247} and~\cite{NewHenry}, and grew out of   realisation that analogous statements hold true for  arbitrary simplicial complexes and not just Vietoris-Rips complexes.  To get these  general statements we use  categorical techniques. This enables us to prove stronger results using    arguments that for us are more transparent.

The most general input for our investigation is a simplicial complex $K$ and a cover $X\cup Y=K_0$ of its set of vertices. 
In this article we study the map $K_X\cup K_Y\subset K$ where $K_X$ and $ K_Y$ are subcomplexes of $K$ consisting of all these simplices of $K$ 
which are subsets of $X$ and  $Y$  respectively.  The goal is to estimate the homotopy fibers of this inclusion. We do that in terms of 
obstruction complexes  $\text{St}(\sigma,X\cap Y):=\{\mu\subset X\cap Y\ |\  0\leq|\mu|\text{ and } \mu\cup \sigma\in K\}$
indexed by  simplices  $\sigma$ in $K$ (see Definition~\ref{asdfsaahdfg}). Our main result, Theorem~\ref{adfsdfhgd},  states that the homotopy  fibers of  $K_X\cup K_Y\subset K$ are in the same cellular class (see Paragraph~\ref{asdfdsfhgdf}, and~\cite{MR1397724,MR1320991,MR1392221})  as the obstruction complexes  $\text{St}(\sigma,X\cap Y)$ for all $\sigma$ in $K$ such that $\sigma\cap X\not=\emptyset$, $\sigma\cap Y\not=\emptyset$, and
$\sigma\cap X\cap Y=\emptyset$. For instance (see Corollary~\ref{sdfdfghsd}.1) if, for all such $\sigma$, the obstruction complex  $\text{St}(\sigma,X\cap Y)$ is contractible, then $K_X\cup K_Y\subset K$ is a weak equivalence and consequently $K$ decomposes  as a homotopy pushout
$\text{hocolim}(K_X\hookleftarrow K_{X\cap Y}\hookrightarrow K_Y)$ leading to  a Mayer-Vietoris exact sequence. Another instance of our result (see Corollary~\ref{sdfdfghsd}.2) states that if these obstruction complexes have trivial homology in degrees not exceeding $n$, then so do 
the homotopy fibers of $K_X\cup K_Y\subset K$ and consequently this map induces an isomorphism on homology in  degrees not exceeding $n$, leading to a partial Mayer-Vietoris exact sequence. Yet another consequence (see Corollary~\ref{sdfdfghsd}.3) is that if these obstruction complexes have $p$-torsion homology for a prime $p$, then, for any field $F$ of characteristic  different than $p$,  the inclusion $K_X\cup K_Y\subset K$ induces an isomorphism on homology with coefficients in $F$ leading again to a Mayer-Vietoris exact sequence.

In section~\ref{dsfgdfhsfghj} we specialise our theorem  about the cellularity of  the homotopy fibers of the  inclusion $K_X\cup K_Y\subset K$ to the case when $K$ is a  clique complex and give some conditions that imply the   assumptions of the theorem  in this case.
Obtained results  in principle generalize all the statements proven in~\cite{MR3824247,NewHenry} for Vietoris-Rips complexes. 
The point we would like to make is that these statements  are not about Vietoris-Rips  complexes but rather about these complexes being clique.
In particular triangular inequality of the input metric space is not needed for our statements to hold.  We then prove  Theorem~\ref{easdfgsdghj} 
for which  it is essential that considered complex is  the Vietoris-Rips complex of the  metric gluing of pseudo-metric spaces for which the triangular inequality is satisfied. This theorem gives 2 connectedness of the relevant homotopy fibers and hence can be used to calculate
 $H_1$ and $H_0$ of the gluing in terms of $H_1$'s and $H_0$'s of the components and the intersection.

\section{Small categories and simplicial sets}\label{asfdhg}
In this section we recall some elements of a convenient language for describing and discussing homotopical  properties of small categories.  The key role here is played by the nerve construction~\cite{MR516607,nerveatnlab}  that transforms small categories into 
simplicial sets. We refer the reader to~\cite{MR0268888,MR1711612} for an overview of how to do homotopy theory on simplicial sets.
We consider the standard model structure on the category of simplicial sets
where weak equivalences are given by the maps inducing  bijections on all the homotopy groups with respect to any choice of a base point.

Here is a list of definitions and characterizations of various homotopical notions for small categories and some
statements regarding these notions.

\begin{point}
Let  $\mathcal C$ be a property of simplicial sets, such as being contractible, $n$-connected, having $p$-torsion integral reduced homology, or having trivial reduced homology in some degrees. 
By definition a small category $I$ satisfies $\mathcal C$ if and only if its  nerve $N(I)$  satisfies  $\mathcal C$.
\end{point}
\begin{point}
Let $\mathcal C$ be a property of maps of simplicial sets, such as being a weak equivalence, a homology isomorphism, or having $n$-connected homotopy fibers.  By definition, a functor $f\colon I\to J$ between small categories satisfies $\mathcal C$ if and only if the map of simplicial sets $N(f)\colon N(I)\to N(J)$ satisfies $\mathcal C$.
\end{point}
\begin{point}
Functors $f,g\colon I\to J$ are homotopic if the maps $N(f),N(g)\colon N(I)\to N(J)$ are homotopic.
For example, if there is a natural transformation $\phi\colon f\Rightarrow g$ between   $f$ and $g$, then $f$ and $g$ are 
homotopic.

Assume $I$ has a terminal object $t$. Then there is a unique natural transformation from the identity
functor $\text{id}\colon I\to I$ to the constant functor $t\colon I\to  I$ with value $t$. 
The identity functor is therefore homotopic to the constant functor, and consequently  $I$ is contractible.
By a similar argument, a category with an initial object is also contractible.
\end{point}

\begin{point}
A commutative square  of small categories  is called  a homotopy push-out (pull-back) if after applying the nerve construction the obtained commutative square of simplicial sets is a homotopy push-out (pull-back). 
\end{point}
\begin{point}\label{asdfdsfhgdf}
Recall that a collection
 $\mathcal C$ of simplicial sets  is  closed if it contains a nonempty simplicial set and  it is closed under weak  equivalencies and   homotopy colimits indexed by arbitrary small contractible categories~\cite[Corollary 7.7]{MR1397724}. Any closed collection contains all contractible simplicial sets~\cite[Proposition 4.5]{MR1397724}. If a closed collection contains an empty    simplicial set, then it contains all simplicial sets.
 
The following are some  examples of collections of simplicial sets that are closed:   contractible simplicial sets, $n$-connected  simplicial sets,  connected simplicial sets having $p$-torsion reduced integral homology, simplicial sets  having trivial reduced  homology with some fixed coefficients up to a given degree, and more generally simplicial sets which are acyclic with respect to some (possibly not ordinary) homology theory.
 
 Let $\mathcal C$ be a closed  collection of simplicial sets and  $f\colon I\to J$ be a functor between small categories. We say that homotopy fibers of $f$ satisfy
 $\mathcal C$ if the homotopy fibers of $N(f)\colon N(I)\to N(J)$, over any component in  $N(J)$,
 belong to $\mathcal C$.
 \end{point}

 \begin{point}
 
Let $f\colon I\to J$ be a functor between small categories. For an object $j $  in $J$,  the symbol $j\!\uparrow\! f$ denotes the category whose objects are  pairs $(i, \alpha\colon j\to f(i))$ consisting of  an object $i$  in $I$ and a morphism  $\alpha\colon j\to f(i)$ in $J$. The set of morphisms in $j\uparrow f$ between $(i, \alpha\colon j\to f(i))$ and
$(i', \alpha'\colon j\to f(i'))$ is by definition the set of morphisms $\beta\colon i\to i'$   in $I$  for which the following triangle commutes:
\[\begin{tikzcd}
& j\ar{dl}[swap]{\alpha}\ar{dr}{\alpha'}\\
f(i)\ar{rr}{f(\beta)} & & f(i')'
\end{tikzcd}\]
The composition in $j\!\uparrow\! f$ is given by the composition in $I$.

For an object $j $  in $J$,  the symbol $f\!\downarrow\! j$ denotes the category whose objects are  pairs $(i, \alpha\colon  f(i)\to j)$ consisting of an object $i$  in $I$ and a morphism  $\alpha\colon f(i)\to  j $ in $J$. The set of morphisms in $f\downarrow j$ between $(i, \alpha\colon  f(i)\to j)$  and
$(i', \alpha'\colon  f(i')\to j)$  is by definition the set of morphisms $\beta\colon i\to i'$   in $I$  for which the following triangle commutes:
\[\begin{tikzcd}
f(i)\ar{rr}{f(\beta)}\ar{dr}{\alpha} & & f(i')\ar{dl}[swap]{\alpha'} \\
& j
\end{tikzcd}\]
The composition in $f\!\downarrow\! j$ is given by the composition in $I$.
\end{point}

\begin{thm}[{\cite[Theorem 9.1]{MR1397724}}]\label{asgdfghfdn}
Let $\mathcal C$ be a closed  collection of simplicial sets and $f\colon I\to J$ be a functor between small categories. 
\begin{enumerate}
\item If, for every  $j$,  $f\!\downarrow\! j$ satisfies   $C$, then so do the homotopy fibers of $f$.
\item  If, for every $j$,   $j\!\uparrow\! f$ satisfies   $C$, then so do the homotopy fibers of $f$.
\end{enumerate}
\end{thm}

Depending on the choice of a closed collection, Theorem~\ref{asgdfghfdn}  leads to:
\begin{cor}\label{asgfgjh}
Let $f\colon I\to J$ be a functor between small categories. 
\begin{enumerate}
\item If,  for every $j$, $f\!\downarrow\! j$  (respectively  $j\!\uparrow\! f$) is contractible, then
$f$  is a weak equivalence.
\item  If,  for every $j$, $f\!\downarrow\! j$  (respectively  $j\!\uparrow\! f$) is $n$-connected for some $n\geq 0$, then
the homotopy fibers of $f$  are $n$-connected. Thus in this case  
$f$ induces an isomorphism on homotopy groups in degrees $0,\ldots,n$ and a surjection in degree $n+1$.
\item  If, for every $j$, $f\!\downarrow\! j$  (respectively  $j\!\uparrow\! f$) is connected and has $p$-torsion reduced integral homology in degrees  not exceeding  $n$  ($n\geq 0$), then
the homotopy fibers of $f$ are connected and have   $p$-torsion reduced integral homology in  degrees  not exceeding  $n$. Thus in this case,  for primes $q\not=p$,
$f$ induces an isomorphism on  $H_\ast(-,{\mathbf Z}/q)$ for  $\ast\leq n$ and a surjection on
$H_{n+1}(-,{\mathbf Z}/q)$.
\item If,  for every $j$, $f\!\downarrow\! j$  (respectively  $j\!\uparrow\! f$)  is acyclic with respect to some homology theory, then $f$ is this homology  isomorphism.
\end{enumerate}
\end{cor}

\section{Simplicial  complexes  and  small categories}
\begin{point}\label{aSFDHFN}
Fix a set $\mathcal U$ called a  {\bf universe}. 
A   {\bf simplicial  complex}  is  a  collection $K$ of finite nonempty subsets  of $\mathcal U$ that  satisfies  the  following  requirement:   if    $\sigma\subset  \mathcal U$  is  in  $K$, then  every    non-empty  subset  of  $\sigma$  is  also    in  $K$.   

Let $X\subset \mathcal U$ be a subset. The collection $\{\{x\}\  |\  x\in  X\}$, consisting of singletons in $X$,  is a simplicial complex denoted  also by $X$, called the  {\bf discrete} simplicial complex on $X$. The collection  $\{\sigma \subset X\ |\ 1\leq |\sigma|<\infty\}$ of all finite nonempty subsets of $X$ is also a simplicial complex  denoted by $\Delta[X]$ and  called the {\bf simplex} on $X$. 
A simplicial complex is called  a standard simplex if it is of the form $\Delta[X]$ for some $X\subset \mathcal U$.
The simplex $\Delta[\emptyset]$ is called the  empty simplex or the empty simplicial complex.
\end{point}

\begin{point}
Let  $K$  be  a  simplicial  complex.
An  element  $\sigma$  in     $K$  is  called  a  simplex  of   $K$ of dimension
$|\sigma|-1$.    The  set  of  $n$-dimensional  simplices in $K$  is  denoted  by  $K_n$.  
An element $x\in \mathcal U$ is called a vertex of $K$ if $\{x\}$ is a simplex in $K$.
The assignment  $x\mapsto \{x\}$ is a bijection between the set of vertices in  $K$ and the set of its
$0$-dimensional simplices $K_0$. We use this bijection to identify these sets. Thus we are going to refer to 
$0$-dimensional simplices in $K$ also as vertices.
\end{point}

\begin{point}\label{asfdfhd}
If  $\{K^{i}\}_{i\in I}$ is a family of simplicial complexes, then
both the intersection $\cap_{i\in I} K^{i}$ and the union  $\cup_{i\in I} K^{i}$ are also simplicial complexes.
If $K$ is a simplicial complex and $X\subset {\mathcal U}$ is a subset, then the intersection
$K\cap \Delta[X]$ is a simplicial complex consisting of the  elements of $K$ that are subsets of $X$.
This intersection is called  the {\bf restriction} of $K$ to  $X$ and is  denoted by $K_X$.

Note that $\Delta[X]\cap \Delta[Y]=\Delta[X\cap Y]$. Thus the  intersection of  standard simplices (see~\ref{aSFDHFN}) is again a standard simplex, which  can possibly be empty.

Let $L$ and $K$ be simplicial complexes.  If $L\subset K$, then
 $L$ is  called a subcomplex of $K$. Being a subcomplex 
is a partial order relation on the collection of all simplicial complexes which  gives this collection the structure of a lattice.  The union is the join  and the intersection is the meet.

The collection $\cup_{0\leq i\leq n} K_i$ is a  subcomplex of $K$   called the   $n$-th {\bf skeleton} of $K$
and denoted by $\text{sk}_nK$.
\end{point}
\begin{point}
  A  {\bf map}  between  
two  simplicial  complexes    $K$  and  $L$  is  by  definition  a  function $\phi\colon K\to L$ for which
there exists  a function   $f\colon  K_0\to  L_0$  such  that  $\phi(\sigma)=\cup_{x\in \sigma }f(\{x\})$  for all $\sigma$ in $K$. 
In particular $\phi(\{x\})=f(\{x\})$  for every vertex  $x$ in $K$.
Thus $f$ is uniquely determined by $\phi$ and 
we often use the symbol $\phi_0$ to denote $f$.  If   $K$ and $L$ are fixed,    then $\phi$ is  
determined by $f=\phi_0$.  
 The inclusion $L\subset K$ of a subcomplex
  is an example of a  map.

 For  any  simplicial  complex  $K$, the inclusions
$K_0\subset  K\subset  \Delta[K_0]$,
  between  the  discrete  simplicial  complex  $K_0$,   $K$, and the simplex $\Delta[K_0]$ on $K_0$
  are maps of simplicial complexes. The induced  functions on the set of vertices  for  these  two  inclusions 
  are given by the identity function  $\text{id}\colon K_0\to K_0$.
    \end{point}
  
  \begin{point}\label{fdfgsfhb} Classically, the geometrical
realization is used to  define and  study homotopical properties of simplicial complexes. 
For example, a commutative square of  simplicial complexes is called a  homotopy push-out (pull-back) if  after applying the realization, the obtained commutative square of spaces is a homotopy push-out (pull-back).
For instance  two simplicial complexes $K$ and $L$  fit into the following commutative diagram of subcomplex 
inclusions:
\[\begin{tikzcd}
 K\cap L \arrow[r, hook] \arrow[d, hook'] &K\arrow[d, hook] \\
 L\arrow[r, hook] & K\cup L
  \end{tikzcd}\]
By applying the realization construction to this square, we obtain a commutative  square of spaces
which is a push-out and hence a homotopy push-out as the maps involved are cofibrations. 
\end{point}

Since the realization of a simplicial complex $K$ can be built from the realization of its $n$-skeleton  $\text{sk}_nK$ by attaching  (possibly in many steps) cells of dimension strictly  bigger than $n$, we get:
\begin{prop}\label{asfsdfhg}
Let $n\geq 0$ be a natural number. For every simplicial complex $L$ such that $\text{\rm sk}_{n+1}K\subset L\subset K$, 
 the homotopy fibers
of the inclusion $L\subset K$ are $n$-connected. In particular, the map $L\subset K$  induces an isomorphism on homotopy  and integral homology groups in    degrees $0,\ldots,n$ and a surjection in degree $n+1$. 
\end{prop}

There are situations however when another way of extracting homotopical properties of simplicial complexes is more convenient.  In the rest of this  section, we recall how one can retrieve and study such information  by first transforming     simplicial complexes into small categories and then    using 
 the nerve construction   as explained in Section~\ref{asfdhg}.

  \begin{point}
 Let   $K$  be a simplicial complex. The {\bf simplex category} of $K$, denoted also by the same symbol $K$, 
 is by definition the inclusion poset of its simplices. Thus, the objects of $K$ are  the simplices in   $K$ and 
 the sets of morphisms are either empty or contain only one element:
  \[|\text{mor}_{K}(\sigma,\tau)|=\begin{cases}
  1  &  \text{  if  }  \sigma\subset  \tau\\
  0  &  \text{  otherwize}
  \end{cases}\]
  
 If $\phi\colon K\to L$ is a map of simplicial complexes, then the assignment $\sigma\mapsto \phi(\sigma)$ is a functor of simplex categories. We denote this functor also by the symbol $\phi\colon K\to L$.
  Not all functors between $K$ and $L$ are of such a form.
    \end{point}

The geometrical realization of a simplicial complex is weakly equivalent to the realization of  the nerve of  this  simplicial complex. Thus  to describe homotopical properties of simplicial complexes we can either
use their   geometrical realizations or the nerves of their simplex categories.

\begin{point}
Let $K$ be a simplicial complex and $\sigma$ be its simplex. Define the {\bf star} of $\sigma$ to be $\text{St}(\sigma):=\{\mu\in K\ |\  \sigma\cup \mu\in K\}$.
Note that $\text{St}(\sigma)$ is a  subcomplex of $K$. The star of any simplex is  contractible. More generally:
\end{point}  

\begin{prop}\label{adgassgfhjg}
Let $\sigma$ be a simplex in  $K$. Then, for any  proper subset $S\subsetneq  \sigma$, the collection 
$L:=\{\mu\in K\ |\ \mu\cap S=\emptyset\text{ and } \sigma\cup \mu\in K\}$ is a contractible simplicial complex
(note that if $S=\emptyset$, then $L=\text{\rm St}(\sigma)$).
\end{prop} 
\begin{proof} For  all $\mu$ in $L$, the inclusions $\mu\hookrightarrow \mu\cup (\sigma\setminus S)\hookleftarrow  \sigma\setminus S$ form natural transformations between:
\begin{itemize}
\item the identity functor $\text{id}\colon L\to L$, $\mu\mapsto \mu$,
\item the constant functor  $ L\to L$, $\mu\mapsto \sigma\setminus S$,
\item and $ L\to L$, given by  $\mu\mapsto  \mu\cup (\sigma\setminus S)$.
\end{itemize}
The identity functor  $\text{id}\colon L\to L$ is therefore  homotopic to the constant functor and consequently $
L$ is  contractible. 
\end{proof}

\begin{point}\label{adghxdfgjhkut}
Let $K$ be a simplicial complex. It's  simplex $\tau$ is called {\bf central} if $K=\text{St}(\tau)$, i.e.,
if for any simplex $\sigma$ in $K$, the set $\sigma\cup\tau$ is also a simplex in $K$.  For example, if $X\subset \mathcal{U}$ is non empty, then all simplices in $\Delta[X]$ (see~\ref{aSFDHFN}) are central.
If $\tau$ is a central simplex in $K$, then so is any subset $\tau'\subset \tau$.
According to Proposition~\ref{adgassgfhjg},  a simplicial complex that has a central simplex is contractible. 
\end{point}
\begin{point}\label{afsaasfagdf}
Let $K$ and $L$ be simplicial complexes. If $K\cap L=\emptyset$, then we define their   {\bf join} $K\ast L$ to be  the simplicial complex
consisting of all subsets of $\mathcal U$ of  the form $\sigma\cup \tau$  where $\sigma$ is in $K$ and $\tau$ is in $L$. 
The join is only defined for disjoint simplicial complexes.  The set of vertices $(K\ast L)_0$
is given by the (disjoint)   union  $K_0\cup L_0$. 
Note that $K\ast \Delta[\emptyset]=K$.
If $\sigma\in L$  is central in $L$, then it is central in $K\ast L$. Thus, for any non-empty  subset $X\subset  \mathcal U\setminus K_0$, the join
$K\ast \Delta[X]$ is contractible. Furthermore the join commutes with unions and intersections: if  $(K_1\cup K_2)\cap L=\emptyset$, then
$(K_1\cup K_2)\ast L=(K_1\ast L)\cup (K_2\ast L)$ and $(K_1\cap K_2)\ast L=(K_1\ast L)\cap (K_2\ast L)$. This can be used to show that,
for any choice of  a base-points in $K$ and $L$, the join $K\ast L$ has the homotopy type of the suspension of the smash $\Sigma(|K|\wedge |L])$.
In particular if $K$ is $n$-connected and $L$ is $m$-connected, then $K\ast L$ is $n+m+1$-connected.
\end{point}

\section{One outside point}\label{afsfdgdsfgh}
In this section we recall how the homotopy type of a simplical complex changes when   a vertex
is added.  We start with defining subcomplexes    that play  an important role in  describing such changes. These complexes are essentially used throughout   the entire paper.
\begin{defn}\label{asdfsaahdfg}
Let $K$ be a simplicial complex and $A\subset \mathcal{U}$ be a subset. 
For a simplex $\sigma$ in $K$, define the \textbf{obstruction complex}:
\[\text{St}(\sigma,A):=\{\mu\subset A\ |\ 0<|\mu| \text{ and }\mu\cup \sigma\in K\}=K_A\cap \text{St}(\sigma) \]
\end{defn}
If $\mu$ belongs to $\text{St}(\sigma,A)$, then so does any of its non empty finite subsets. Thus $\text{St}(\sigma,A)$ is a simplicial complex. It is a subcomplex of  $K_A$.
Note that the complex  $\text{St}(\sigma,A)$ may be  empty.   If $\tau\subset \sigma$, then    $\text{St}(\sigma,A)\subset \text{St}(\tau,A)$. 

Fix a vertex $v$ in $K$. Any simplex in $K$ either  contains $v$ or it does not.
This means    $K=K_{K_0\setminus\{v\}}\cup \text{St}(v)$ and hence we have a homotopy push-out square:


\[\begin{tikzcd}
 K_{K_0\setminus\{v\}}\cap  \text{St}(v) \ar[hook]{r} \ar[hook']{d} & \text{St}(v)\ar[hook]{d}\\
K_{K_0\setminus\{v\}}\ar[hook]{r}  & K
\end{tikzcd}\]
 By definition $ K_{K_0\setminus\{v\}}\cap  \text{St}(v)=\text{St}(v,K_0\setminus\{v\})$.   Proposition~\ref{adgassgfhjg} gives contractibility of  $\text{St}(v)$. The simplicial complex $K$ fits therefore  into the following homotopy cofiber sequence:
\[\text{St}(v,K_0\setminus\{v\})\hookrightarrow    K_{K_0\setminus\{v\}}\hookrightarrow  K\]
Here are some basic consequences of this fact:
\begin{cor} Let $v$ be a vertex in a simplical complex $K$.
\begin{enumerate}
\item If  $\text{\rm St}(v,K_0\setminus\{v\})$ is contractible, then
$K_{K_0\setminus\{v\}}\subset K$  is a weak equivalence.
\item  If  $\text{\rm St}(v,K_0\setminus\{v\})$ is $n$-connected for a natural number $n\geq 0$, then the map 
$K_{K_0\setminus\{v\}}\subset K$ induces an isomorphism on homotopy groups in degrees $0,\ldots,n$ and a surjection in degree $n+1$.
\item  If  $\text{\rm St}(v,K_0\setminus\{v\})$   is connected and has $p$-torsion reduced integral homology in degrees  not exceeding  $n$  ($n\geq 0$),  then    for a prime $q$ not dividing $p$,
$K_{K_0\setminus\{v\}}\subset K$  induces an isomorphism on  $H_\ast(-,{\mathbf Z}/q)$ for  $\ast\leq n$ and a surjection on
$H_{n+1}(-,{\mathbf Z}/q)$.
\item If  $\text{\rm St}(v,K_0\setminus\{v\})$   is acyclic with respect to some homology theory, then
$K_{K_0\setminus\{v\}}\subset K$ 
is this homology  isomorphism.
\end{enumerate}
\end{cor}

\section{Two outside points}\label{asfasfgddf}
Let us fix two distinct vertices $v_0$ and $v_1$ in a simplicial complex $K$.
Note that $\left(K_0\setminus\{v_0\}\right)\cup \left(K_0\setminus\{v_1\}\right)=K_0$. In this section we are going to
investigate the inclusion $K_{K_0\setminus\{v_0\}}\cup K_{K_0\setminus\{v_1\}}\subset K$.

There are two possibilities. First, $\{v_0,v_1\}$
is not a simplex in $K$. In this case  $K_{K_0\setminus\{v_0\}}\cup K_{K_0\setminus\{v_1\}}=K$.

Assume  
$\{v_0,v_1\}$ is a simplex in $K$.
 Then: 
 \[K=K_{K_0\setminus\{v_0\}}\cup K_{K_0\setminus\{v_1\}}\cup \text{St}(v_0,v_1)\]
 Consequently  there is a homotopy push-out square (see~\ref{fdfgsfhb}):
\[\begin{tikzcd}
\left(K_{K_0\setminus\{v_0\}}\cup K_{K_0\setminus\{v_1\}}\right)\cap \text{St}(v_0,v_1) \arrow[d, hook'] \arrow[r, hook] &\text{St}(v_0,v_1)\arrow[d, hook]\\
K_{K_0\setminus\{v_0\}}\cup K_{K_0\setminus\{v_1\}}\arrow[r, hook] & K
\end{tikzcd}\]
Since the star complex $\text{St}(v_0,v_1)$ is contractible (see Proposition~\ref{adgassgfhjg}),
  $K$ is  therefore weakly equivalent to the  homotopy cofiber of the map:
\[\left(K_{K_0\setminus\{v_0\}}\cup K_{K_0\setminus\{v_1\}}\right)\cap \text{St}(v_0,v_1)\hookrightarrow K_{K_0\setminus\{v_0\}}\cup K_{K_0\setminus\{v_1\}}\]
The complex $\left(K_{K_0\setminus\{v_0\}}\cup K_{K_0\setminus\{v_1\}}\right)\cap \text{St}(v_0,v_1)$ fits into the following homotopy push-out square:
\[\begin{tikzcd}
K_{K_0\setminus\{v_0\}}\cap K_{K_0\setminus\{v_1\}}\cap \text{St}(v_0,v_1)\arrow[r, hook] \arrow[d, hook']& K_{K_0\setminus\{v_0\}}\cap  \text{St}(v_0,v_1) \arrow[d, hook]\\
K_{K_0\setminus\{v_1\}}\cap  \text{St}(v_0,v_1) \arrow[r, hook] & \left(K_{K_0\setminus\{v_0\}}\cup K_{K_0\setminus\{v_1\}}\right)\cap \text{St}(v_0,v_1)
\end{tikzcd}\]
Let us identify the complexes in this square:
\begin{itemize}
\item $K_{K_0\setminus\{v_0\}}\cap  \text{St}(v_0,v_1)=\{\mu\in K\ |\ \mu\cap \{v_0\}=\emptyset\text{ and } \{v_0,v_1\}\cup \mu\in K\}$ and thus according to
Proposition~\ref{adgassgfhjg} this complex is contractible;
\item by the same argument $K_{K_0\setminus\{v_1\}}\cap  \text{St}(v_0,v_1)$ is also contractible;
\item $K_{K_0\setminus\{v_0\}}\cap K_{K_0\setminus\{v_1\}}\cap \text{St}(v_0,v_1)=K_{K_0\setminus\{v_0,v_1\}}\cap \text{St}\left(v_0,v_1\right)=\\=
 \text{St}\left(\{v_0,v_1\},K_0\setminus\{v_0,v_1\}\right)$
\end{itemize}
It follows that  $ \left(K_{K_0\setminus\{v_0\}}\cup K_{K_0\setminus\{v_1\}}\right)\cap \text{St}(v_0,v_1)$
has the homotopy type of the suspension of the obstruction complex $\text{St}:=\text{St}\left(\{v_0,v_1\},K_0\setminus\{v_0,v_1\}\right)$
and hence we have a  homotopy
cofiber sequence of the form:
\[\Sigma \text{\rm St}\to K_{K_0\setminus\{v_0\}}\cup K_{K_0\setminus\{v_1\}}\hookrightarrow K\]


\section{$n+1$ outside points}
Homotopy cofiber sequences  described in Sections~\ref{afsfdgdsfgh} and~\ref{asfasfgddf} are particular cases of a more general statement
regarding an arbitrary number 
of outside points.  The aim of this section is to present this generalization.

Let us fix  a set   $\sigma=\{v_0,v_1,\ldots,v_n\}\subset K_0$ of $n+1$ distinct vertices in a simplcial complex $K$
which may not necessarily be a simplex in $K$.
Note that  $\bigcup_{v\in \sigma } \left(K_0\setminus\{v\}\right)=K_0$. In this section we are going to investigate  the inclusion
$\left(\bigcup_{v\in\sigma} K_{K_0\setminus\{v\}}\right) \subset K$

There are two possibilities.
 First, $\sigma$ is not a simplex in $K$. In this case
$\bigcup_{v\in \sigma} K_{K_0\setminus\{v\}}= K$.
 
 Assume $\sigma$ is a simplex in $K$. Then: 
 \[K=\left(\bigcup_{v\in \sigma} K_{K_0\setminus\{v\}}\right) \cup \text{St}(\sigma)\]
 Consequently  there is a homotopy push-out square (see~\ref{fdfgsfhb}):
\[\begin{tikzcd}
\left(\bigcup_{v\in \sigma} K_{K_0\setminus\{v\}}\right)\cap \text{St}( \sigma) \arrow[d, hook'] \arrow[r, hook] &\text{St}( \sigma)\arrow[d, hook]\\
\bigcup_{v\in  \sigma} K_{K_0\setminus\{v\}}\arrow[r, hook] & K
\end{tikzcd}\]
Since the star complex $\text{St}( \sigma)$ is contractible (see Proposition~\ref{adgassgfhjg}),
  $K$ is  therefore weakly equivalent to the  homotopy cofiber of the map:
\[\left(\bigcup_{v\in  \sigma} K_{K_0\setminus\{v\}}\right)\cap \text{St}( \sigma)\hookrightarrow \bigcup_{v\in  \sigma} K_{K_0\setminus\{v\}}\]
Next we identify the homotopy type of  $\left(\bigcup_{v\in  \sigma} K_{K_0\setminus\{v\}}\right)\cap \text{St}( \sigma)$:

\begin{prop}\label{asfdgdsfjdhhg}
Let  $\sigma$ be  a simplex of dimension $n$ in a simplicial complex $K$. Then  $\left(\bigcup_{v\in  \sigma} K_{K_0\setminus\{v\}}\right)\cap \text{\rm St}( \sigma)$ 
has the homotopy type of the $n$-th suspension of the obstruction complex $\Sigma^n \text{\rm St}( \sigma, K_0\setminus\sigma)$.
\end{prop}
\begin{proof}
Consider the inclusion poset of all  subsets $\tau\subset \sigma$. For any such subset $\tau\subset \sigma$,
define:
\[F(\tau):=\begin{cases}
\bigcap_{v\in \tau} K_{K_0\setminus\{v\}}= K_{K_0\setminus\tau} & \text{ if } \tau\not=\emptyset\\
\bigcup_{v\in  \sigma} K_{K_0\setminus\{v\}} & \text{ if } \tau=\emptyset
\end{cases}\]
Note that if $\tau'\subset\tau\subset\sigma$, then
$F(\tau)\subset F(\tau')$. Thus by assigning to the inclusion $\tau'\subset\tau$ the map
$F(\tau)\subset F(\tau')$, we obtain a contra-variant functor indexed by the inclusion poset of all subsets of $\sigma$. For example in the case
$\sigma=\{v_0,v_1,v_2\}$, this contra-variant functor describes a commutative cube:
\[\begin{tikzcd}[column sep=0.1em,row sep=2em]
K_{K_0\setminus\{v_0,v_1,v_2\}}\ar[hook]{rr}\ar[hook]{dd}\ar[hook]{dr} & & K_{K_0\setminus\{v_1,v_2\}}\ar[hook]{dr}\ar[hook]{dd}
\\
&K_{K_0\setminus\{v_0,v_2\}}\ar[hook, crossing over]{rr} & & K_{K_0\setminus\{v_2\}}\ar[hook]{dd}
\\
K_{K_0\setminus\{v_0,v_1\}}\ar[hook]{rr}\ar[hook]{dr} & & K_{K_0\setminus\{v_1\}}\ar[hook]{dr}
\\
& K_{K_0\setminus\{v_0\}}\ar[hook]{rr}\ar[hook, crossing over,from=uu] & & K_{K_0\setminus\{v_0\}}\cup K_{K_0\setminus\{v_1\}}\cup K_{K_0\setminus\{v_2\}} 
\end{tikzcd}\]

For arbitrary $n$, the functor $F$ describes a commutative cube of dimension $n+1$.   This cube is both co-cartesian and  strongly cartesian
 (\cite{}). It is therefore also a  homotopy co-cartesian. By intersecting with $\text{St}(\sigma)$, we obtain a new cube
$\tau\mapsto F(\tau)\cap \text{St}(\sigma)$.  The properties of being co-cartesian and  strongly cartesian are preserved by taking such interection.
Consequently $\left(\bigcup_{v\in  \sigma} K_{K_0\setminus\{v\}}\right)\cap \text{\rm St}( \sigma)$ has the homotopy type of
$\text{hocolim}_{\emptyset\not=\tau\subset\sigma}\left(K_{K_0\setminus\tau}\cap \text{\rm St}( \sigma)\right)$.

For any proper subset $\emptyset\not=\tau\subsetneq \sigma$, we have an equality:
\[K_{K_0\setminus\tau}\cap \text{\rm St}( \sigma)=\{\mu\ |\ \mu\cap \tau=\emptyset \text{ and } \sigma\cup\mu\in K\}\]
We can then use Proposition~\ref{adgassgfhjg} to conclude that $K_{K_0\setminus\tau}\cap \text{\rm St}( \sigma)$ is contractible
if $\emptyset\not=\tau\subsetneq \sigma$. Thus all the spaces in the cube $\tau\mapsto F(\tau)\cap \text{St}(\sigma)$, except for the initial and the terminal, are contractible. That implies that the terminal space $F(\emptyset)\cap \text{St}(\sigma)=\left(\bigcup_{v\in \sigma} K_{K_0\setminus\{v\}}\right)\cap \text{St}(\sigma)$  is homotopy equivalent to the $n$-th suspension of the initial space: $\Sigma^n  \left(F(\sigma)\cap \text{St}(\sigma)
\right)=
\Sigma^n \left( K_{K_0\setminus \sigma}\cap \text{St}(\sigma)\right)=\Sigma^n\text{St}(\sigma,K_0\setminus \sigma)$.
\end{proof}

We finish this section with summarising the  consequences of  the discussion leading to Proposition~\ref{asfdgdsfjdhhg} and  the proposition itself:
\begin{cor}\label{asfgdsgfjghkkl}
 Let $\sigma\subset K_0$ be a subset consisting of  $n+1$ distinct vertices in a simplicial complex $K$. 
\begin{enumerate}
\item If  $\sigma $ is not a simplex in $K$, then  $\bigcup_{v\in\sigma}K_{K_0\setminus\{v\}}=K$.
\item Assume $\sigma$  is a simplex in $K$.
\begin{enumerate}
\item Then there is a  homotopy
cofiber sequence:
\[\Sigma^n\text{\rm St}(\sigma,K_0\setminus \sigma)\to \bigcup_{v\in\sigma}K_{K_0\setminus\{v\}}\hookrightarrow K\]
\item  If $\text{\rm St}(\sigma,K_0\setminus \sigma)=\emptyset$, then   there is a  homotopy
cofiber sequence (here $S^{-1}=\emptyset$):
\[S^{n-1}\to  \bigcup_{v\in\sigma}K_{K_0\setminus\{v\}}\hookrightarrow K\]
\item If $\text{\rm St}(\sigma,K_0\setminus \sigma)\not =\emptyset$, then  the homotopy fibers of   $\bigcup_{v\in\sigma}K_{K_0\setminus\{v\}}\hookrightarrow K$ are $ m\geq 0$ connected if and only if  $\overline{H}_i(\text{\rm St}(\sigma,K_0\setminus \sigma),{\mathbf Z})=0$
for $i\leq m-n$.
\item  If $\text{\rm St}(\sigma,K_0\setminus \sigma)\not =\emptyset$, then $\bigcup_{v\in\sigma}K_{K_0\setminus\{v\}}\hookrightarrow K$ is a weak
equivalence if and only if  $\overline{H}_i\left(\text{\rm St}(\sigma,K_0\setminus \sigma),{\mathbf Z}\right)=0$ for all $i$.
\end{enumerate}
\end{enumerate}
\end{cor}

\section{Push-out decompositions I.}\label{dec}
In this section our starting assumption is:

\begin{point}[\bf Starting input I]\label{adfdfhgf}
{\em $K$ is a  simplicial complex, $X\cup Y=K_0$ is a cover of its set of vertices, and  $A:=X\cap Y$.}
\end{point}

By restricting $K$ to $X$ and $Y$, and taking the union of these restrictions we obtain a subcomplex $K_X\cup K_Y\subset K$. 
Since $K_X\cap K_Y=K_A$, this subcomplex fits into the following homotopy push-out square:
\[\begin{tikzcd}
K_A\ar[hook]{r} \ar[hook']{d} & K_X \ar[hook]{d}\\
K_Y\ar[hook]{r} & K_X\cup K_Y
\end{tikzcd}\]
This push-out can be then used to extract  various  homotopical properties of the union $K_X\cup K_Y$  
from the  properties of $K_X$, $K_Y$ and $K_A$. For example, 
 if $K_X$, $K_Y$ and $K_A$ belong to a closed collection
(see~\ref{asdfdsfhgdf}), then so does $K_X\cup K_Y$.
If $K_A$ is contractible, then $K_X\cup K_Y$
has the homotopy type of the wedge of  $K_X$ and $K_Y$, and its  reduced homology  is the sum of the reduced homologies of $K_X$ and $K_Y$.
 More generally, there is a Mayer-Vietoris  sequence
connecting homologies of $K_X\cup K_Y$ with those of $K_X$, $K_Y$ and $K_A$.

A fundamental question discussed in this article  is:  {\em  under what circumstances the inclusion 
$K_X\cup K_Y\subset K$ is a weak equivalence,
or homology isomorphism, or has highly connected homotopy fibers etc?} Such  circumstances would  enable us to express various  homotopical properties of $K$ in terms of the   properties of its restrictions
$K_X$, $K_Y$ and $K_A$.

\begin{defn}\label{asdgfssdtjhgkj}
Under the starting assumption~\ref{adfdfhgf}, define $P$ to be the  subposet of  $K$
given by:
\[P:=\{\sigma\in K\ |\ \sigma\subset X\text{ or } \sigma\subset Y \text{ or }  \sigma\cap A\not=\emptyset\}\]
\end{defn}

We are going to be  more interested in the set of simplices of $K$ that do not belong to $P$, which explicitly can be described as:
\[K\setminus P=\{\sigma\in K\ |\ \sigma\cap X\not=\emptyset\text{ and } \sigma\cap Y\not=\emptyset
\text{ and } \sigma\cap A=\emptyset\}\]

The poset  $P$  may not be  the simplex category of any simplicial complex.
There are  two poset inclusions   that we denote by $f$ and $g$:

\[\begin{tikzcd}
 K_X\cup K_Y \arrow[r, hook, "f"] & P \arrow[r, hook,"g"] & K
 \end{tikzcd}\]

Our  first general observation is:
\begin{prop}\label{asfsdfg}
The functor $f\colon K_X\cup K_Y\hookrightarrow  P$ is a weak equivalence.
\end{prop}
\begin{proof}
We are going to show that, for every $\sigma$ in $P$, $f\!\downarrow\! \sigma$ is  contractible.

First assume
 $\sigma\subset X$ or $\sigma\subset Y$. Then the object $(\sigma, \text{id}\colon\sigma\to \sigma)$ is terminal in $f\!\downarrow\! \sigma$ and consequently 
this category  is contractible. 

Assume    $\sigma\cap A\not=\emptyset$.  Then, for any 
object $(\tau,\tau\subset \sigma)$ in $f\!\downarrow\! \sigma$, the  subsets $\tau$, $\tau\cup (\sigma\cap A)$, and $ \sigma\cap A$ of $\sigma$ are simplices that belong to $K_X\cup K_Y$. We can then form the following commutative diagram in $P$  
  where the top horizontal arrows represent  morphisms in    $K_X\cup K_Y$:
  \[\begin{tikzcd}
 \tau \arrow[r, hook]  \arrow[rd, hook] &\tau\cup (\sigma\cap A) \arrow[d, hook] & \sigma\cap A  \arrow[l, hook'] \arrow[dl, hook']  \\
& \sigma
  \end{tikzcd}\]
These horizontal morphisms form  natural transformations between:
\begin{itemize}
\item 
the identity functor
$\text{id}\colon f\!\downarrow\! \sigma\to f\!\downarrow\! \sigma$, $(\tau,\tau\subset \sigma)\mapsto(\tau,\tau\subset \sigma)$,
\item the constant functor $ f\!\downarrow\! \sigma\to f\!\downarrow\! \sigma$, $(\tau,\tau\subset \sigma)\mapsto ( \sigma\cap A, \sigma\cap A\subset \sigma)$,
\item and 
$ f\!\downarrow\! \sigma\to f\!\downarrow\! \sigma$ given by  $(\tau,\tau\subset \sigma)\mapsto (\tau\cup( \sigma\cap A), \tau\cup( \sigma\cap A)\subset \sigma)$.
\end{itemize}
The identity functor $\text{id}\colon f\!\downarrow\! \sigma\to f\!\downarrow\! \sigma$ 
is therefore homotopic to the constant functor. This can happen only if $f\!\downarrow\! \sigma$ is a contractible category.
\end{proof}

According to Proposition~\ref{asfsdfg}, the homotopy fibers of $g\colon P\subset K$ and the inclusion $K_X\cup K_Y\subset K$ 
   are weakly equivalent. To understand these homotopy fibers, we are going to focus on the categories  $\sigma\!\uparrow \! g$ and then utilise Corollary~\ref{asgfgjh}.  The functor 
$\sigma\mapsto \sigma\!\uparrow \! g$  fits into the following diagram of natural transformations between functors 
 indexed by $K^{\text{op}}$ with small categories as values:
\[\begin{tikzcd}
& \sigma\ar[mapsto, bend right]{dl}\ar[mapsto]{d}\ar[mapsto, bend left]{dr}\\
\text{St}(\sigma,A)\ar{r}{\psi_{\sigma}}& \sigma\!\uparrow \! g 
\ar{r}{\phi_{\sigma}} & \text{St}(\sigma)
\end{tikzcd}
\]
where:
\begin{itemize}
\item 
 $\psi_{\sigma}\colon \text{St}(\sigma,A)\to \sigma\!\uparrow \! g$ assigns to 
$\mu$ in $\text{St}(\sigma,A)$ the object  in $ \sigma\!\uparrow \! g$ given by the pair  $\psi_{\sigma}(\mu):=(\mu\cup \sigma,  \sigma\subset \mu\cup \sigma)$. 
\item $\phi_{\sigma}\colon \sigma\!\uparrow \! g\to \text{St}(\sigma)$ assigns to
$(\tau,\sigma\subset \tau)$  the simplex  $\tau$ in $ \text{St}(\sigma)$.
\end{itemize}

These natural transformations satisfy  the following properties:
\begin{prop}\label{adfgsdsgfjdh}
Let  $\sigma$ be a simplex in $K$.
\begin{enumerate}
\item If $\sigma$ is in $ P$, then $\sigma\!\uparrow \! g$
is contractible and $\phi_{\sigma}\colon \sigma\!\uparrow \! g\to \text{\rm St}(\sigma)$  is a weak equivalence.
\item If $\sigma$ is in $K\setminus P$,
then $\psi_{\sigma}\colon \text{\rm St}(\sigma,A)\to \sigma\!\uparrow \! g$  is a weak equivalence.
\end{enumerate}
\end{prop}
\begin{proof}
If $\sigma$ is in $P$, then $(\sigma,\text{id}\colon \sigma\to \sigma)$ is an initial object in $\sigma\!\uparrow \! g$ and hence
this category is contractible.  That proves (1).

Assume $\sigma$ is not in $P$, which  is equivalent to $\sigma\cap Y\not=\emptyset$ and $ \sigma\cap X\not=\emptyset$ and $\sigma\cap A=\emptyset$.
Let $(\tau,\sigma\subset \tau)$ be an object  in  $\sigma\!\uparrow \! g$. Define ${\alpha_{\sigma}}(\tau,\sigma\subset \tau):=\tau\cap A$.
Since  $\sigma\cap Y\not=\emptyset$ and $ \sigma\cap X\not=\emptyset$, then $\tau\cap Y\not=\emptyset$ and $ \tau\cap X\not=\emptyset$.
This together with the fact that  $\tau$ belongs to $P$ implies $\alpha_{\sigma}(\tau,\sigma\subset \tau)=\tau\cap A\not=\emptyset$.
Furthermore $(\tau\cap A)\cup \sigma\subset \tau\in P\subset K$. Thus $\alpha_{\sigma}$ defines a functor $\alpha_{\sigma}\colon \sigma\!\uparrow \! g\to \text{St}(\sigma,A)$.
 Note:
 \[\alpha_{\sigma}\psi_{\sigma}(\mu)=\alpha_{\sigma}(\mu\cup \sigma,  \sigma\subset \mu\cup \sigma)=(\mu\cup \sigma)\cap A\]
 Since $\sigma\cap A=\emptyset$ and $\mu\subset A$, we get $\alpha_{\sigma}\psi_{\sigma}(\mu)=(\mu\cup \sigma)\cap A=\mu$. The composition 
 $\alpha_{\sigma}\psi_{\sigma}$ is therefore the identity functor.
 
 Note further:
 \[\psi_{\sigma}\alpha_{\sigma}(\tau,\sigma\subset \tau)=\psi_{\sigma}(\tau\cap A)=\left((\tau\cap A)\cup\sigma,\sigma\subset(\tau\cap A)\cup\sigma\right) \]
 Since $\sigma\subset \tau$, we have a commutative diagram:
 \[\begin{tikzcd}
 &  \sigma\arrow[dl, hook']\arrow[dr, hook]\\
 (\tau\cap A)\cup\sigma \arrow[rr, hook] & & \tau
 \end{tikzcd}\]

 The bottom horizontal morphisms  form a natural transformation between:
 \begin{itemize}
 \item the composition  $ \psi_{\sigma}\alpha_{\sigma}\colon \sigma\!\uparrow \! g\to \sigma\!\uparrow \! g$ and
 \item the identity functor $\text{id}\colon  \sigma\!\uparrow \! g\to \sigma\!\uparrow \! g$.
 \end{itemize}
 The functor  $ \psi_{\sigma}\colon \text{St}(\sigma,A)\to \sigma\!\uparrow \! g$ has therefore a homotopy inverse and hence 
 is  a weak equivalence which proves (2).
 \end{proof}

We use Corollary~\ref{asgfgjh} and Proposition~\ref{adfgsdsgfjdh} to obtain our main statement describing properties of the homotopy fibers of the inclusion $K_X\cup K_Y\subset K$:

\begin{thm}\label{adfsdfhgd}
Notation as in~\ref{adfdfhgf} and Definition~\ref{asdgfssdtjhgkj}.
Let ${\mathcal C}$ be a closed collection of simplicial sets (see~\ref{asdfdsfhgdf}).
Assume that,  for every  $\sigma$ in $K\setminus P$,
the obstruction complex   $\text{\rm St}(\sigma,A)$  (see~\ref{asdfsaahdfg}) satisfies ${\mathcal C}$. Then the homotopy fibers of the inclusion $K_X\cup K_Y\subset K$ also satisfy 
${\mathcal C}$.
\end{thm}

The following are some particular cases of the above theorem specialized to different closed collections of  simplicial sets.
\begin{cor}\label{sdfdfghsd} 
Notation as in~\ref{adfdfhgf} and~\ref{asdgfssdtjhgkj}.  Let $n$ be a natural number.
\begin{enumerate}
\item \label{afgds} If,  for every $\sigma$ in $ K\setminus P$ (see~\ref{asdgfssdtjhgkj}), the simplicial complex $\text{\rm St}(\sigma,A)$ (see~\ref{asdfsaahdfg}) is contractible, then
$K_X\cup K_Y\subset K$  is a weak equivalence.
\item  If,  for every $\sigma$ in $ K\setminus P$, the simplicial complex $\text{\rm St}(\sigma,A)$ is $n$-connected, then
the homotopy fibers of $K_X\cup K_Y\subset K$  are $n$-connected and this map induces  
an isomorphism on homotopy groups in degrees $0,\ldots,n$ and a surjection in degree $n+1$.
\item  Let $p$ be a prime  number. If,  for every $\sigma$ in $ K\setminus P$, the simplicial complex $\text{\rm St}(\sigma,A)$  is connected and has $p$-torsion reduced integral homology in degrees  not exceeding  $n$, then
the homotopy fibers of  $K_X\cup K_Y\subset K$  are connected and have   $p$-torsion reduced integral homology in  degrees  not exceeding  $n$. Thus in this case,  for  prime $q\not=p$,
$K_X\cup K_Y\subset K$  induces an isomorphism on  $H_\ast(-,{\mathbf Z}/q)$ for  $\ast\leq n$ and a surjection on
$H_{n+1}(-,{\mathbf Z}/q)$.
\item If,  for every $\sigma$ in $ K\setminus P$, the simplicial complex $\text{\rm St}(\sigma,A)$   is acyclic with respect to some homology theory, then
$K_X\cup K_Y\subset K$ 
is this homology  isomorphism.
\end{enumerate}
\end{cor}

We remark that  Corollary~\ref{sdfdfghsd}.\ref{afgds}  is a generalization of ~\cite[Theorem 2]{NewHenry} to  abstract simplicial complexes.

Requirements for obtaining  $n$-connected fibers can be weakened:
\begin{prop}\label{sfsfdgfdhfg}
 Notation as in~\ref{adfdfhgf}  and~\ref{asdgfssdtjhgkj}. Let $n$ be a natural number.
 If  $\text{\rm St}(\sigma, A)$ is $n$-connected  for every $\sigma$ in $(\text{\rm sk}_{n+1}K)\setminus P$, then the homotopy fibers of $K_X\cup K_Y\subset K$ are $n$-connected.
\end{prop}
\begin{proof}
Consider the following  poset  inclusions:
\[\begin{tikzcd}
K_X\cup K_Y\ar[hook]{r}{f}  &  P \ar[hook]{r}{g_1}\ar[hook,bend right=25]{rr}{g}  &  P\cup \text{sk}_{n+1}K \ar[hook]{r}{g_2} & K
\end{tikzcd}\]
According to Proposition~\ref{asfsdfg},
$f$ is a weak equivalence.  The homotopy fibers of $g_2$ are $n$-connected
 by Proposition~\ref{asfsdfhg}. Thus if the homotopy fibers of $g_1$ are $n$-connected, then so are 
 the 
  homotopy fibers of the inclusion $K_X\cup K_Y\subset K$. To show that 
 the homotopy fibers of $g_1$ are $n$-connected it is enough to show that the categories 
 $\sigma\!\uparrow \! g_1$ are $n$-connected for every  $\sigma$ in $P\cup \text{sk}_{n+1}K$.
 Proposition~\ref{adfgsdsgfjdh} gives that
 $\sigma\!\uparrow \! g_1$ is  contractible  if $\sigma$  is in $P$,  and   is weakly equivalent to $\text{\rm St}(\sigma,A)$  if  $\sigma$ is in $\text{sk}_{n+1}K\setminus P$.  By the assumption $\text{\rm St}(\sigma,A)$ are therefore   $n$-connected.
 \end{proof}

\section{Push-out decompositions II.}\label{asfdsdfhiu}
Theorem~\ref{adfsdfhgd}  states that the homotopy fibers of the inclusion
$K_X\cup K_Y \subset K$  belong to the smallest closed collection containing  all the  complexes $\text{St}(\sigma,A)$
for  $\sigma$ in $K\setminus P$.   
Recall that if a closed collection contains an empty simplicial set, then it contains all simplicial  sets, in which case  
Theorem~\ref{adfsdfhgd} has no content.  Thus   $\text{St}(\sigma,A)$   being non empty, for all
 $\sigma$  in $K\setminus P$, is an absolute  minimum requirement for Theorem~\ref{adfsdfhgd} to have any content.  
In  most of  our  statements that follow,  the  assumptions we make have  much stronger global non emptiness consequences of the form:
\[\bigcap_{\sigma\in K\setminus P}\text{St}(\sigma,A)\not= \emptyset\ \ \ \ \ \ \ \bigcap_{\sigma\in K_{n+1}\setminus P}\text{St}(\sigma,A)\not= \emptyset
\ \ \ \ \ \ \ \bigcap_{\sigma\in (\text{sk}_{n+1}K)\setminus P}\text{St}(\sigma,A)\not= \emptyset\]
Here is a consequence of  having one of these intersections non-empty:
\begin{prop}\label{adgsfgjkk}
Notation as in~\ref{adfdfhgf} and~\ref{asdgfssdtjhgkj}. Assume: 
\[\bigcap_{\sigma\in K_{1}\setminus P}\text{\rm St}(\sigma,A)\not= \emptyset\]
Then the homotopy fibers of $K_X\cup K_Y\subset K$  are connected. 
\end{prop}
\begin{proof}
Let  $v$ be a vertex  in $\bigcap_{\sigma\in K_{1}\setminus P}\text{St}(\sigma,A)$. 
Observe that
$\text{sk}_{1}(K)$ is a disjoint union of  $K_1\setminus P$ and  $\text{sk}_{1}(K)\cap P$.
This can fail   for $\text{sk}_{n}(K)$ if $n>1$. 
For every $\tau$ in $\text{sk}_{1}(K)$, define:
\[\phi(\tau):=\begin{cases}\tau\cup{v} & \text{ if } \tau\in K_1\setminus P \\ 
\tau &\text{ if } \tau\in \text{sk}_{1}(K)\cap P
\end{cases}\]
If $\tau\subsetneq \tau'$ in  $\text{sk}_{1}(K)$, then $\tau$ is in  $P$ and hence   $\tau=\phi(\tau)\subset \phi(\tau')$.
In this way we obtain a functor $\phi\colon \text{sk}_{1}(K)\to P$.  The inclusion $\tau\subset \phi(\tau)$,
is a natural transformation between the skeleton inclusion $\text{sk}_{1}(K)\subset K$ and the composition:
\[\begin{tikzcd}
 \text{sk}_{1}(K) \arrow[r, hook, "\phi"] & P \arrow[r, hook,"g"] & K
 \end{tikzcd}\]
Thus these two functors from $\text{sk}_{1}(K)$ to $ K$ are homotopic. The statement of the proposition
is then a consequence of
Proposition~\ref{asfsdfhg}. 
\end{proof}

Proposition~\ref{adgsfgjkk} does not generalise to $n>0$. Non-emptiness of the intersection 
$\bigcap_{\sigma\in (\text{sk}_{n+1}K)\setminus P}\text{St}(\sigma,A)$ does not imply that
the homotopy  fibers of  $K_X\cup K_Y\subset K$ are $n$-connected. For an easy example see~\ref{basfasfgdqweter}.
To guarantee $n$-connectedness of these homotopy fibers we need additional restrictions. For example in the following corollary the assumptions imply  that $\text{St}(\sigma,A)$ does not depend on $\sigma$ in $(\text{\rm sk}_{n+1}K)\setminus P$:

\begin{cor}\label{asfdfgjfgjk,}
 Notation as in~\ref{adfdfhgf} and~\ref{asdgfssdtjhgkj}.  Let $n$ be a natural number.
 Assume that one of the following  conditions is satisfied:
 \begin{enumerate}
 \item There is an $n$-connected simplicial complex  $L$ such that, for every  simplex $\sigma$ in $(\text{\rm sk}_{n+1}K)\setminus P$, $\text{\rm St}(\sigma,A)=L$.
 \item The complex $K_A$ is $n$-connected and, for every simplex $\sigma$ in   $(\text{\rm sk}_{n+1}K)\setminus P$, $\text{\rm St}(\sigma,A)=K_A$.
 \item  The set $A$ is non empty. Furthermore, for every simplex $\sigma$ in   $(\text{\rm sk}_{n+1}K)\setminus P$ and every finite subset $\mu$ in $A$, the union $\sigma\cup \mu$ is a simplex in $K$.
 \item $A=\{v\}$ and, for every simplex $\sigma$ in   $(\text{\rm sk}_{n+1}K)\setminus P$,   the  union $\sigma\cup\{v\}$ is also a simplex in $K$.
 \end{enumerate}
Then the homotopy fibers of  $K_X\cup K_Y\subset K$ are $n$-connected.
\end{cor}
\begin{proof}
The corollary under the assumption (1) is a  direct consequence of Proposition~\ref{sfsfdgfdhfg}.
The assumption (2) is a  particular case of (1) with $L=K_A$. The assumption (3) is a particular case of (1) with
$L=\Delta[A]$. Finally,  the assumption (4)  is a particular case of (3).
\end{proof}

Here is another example of a statement whose assumption,  referred to as  ``one entry point", has  a global nonemptiness consequence:

\begin{cor}\label{afgsfdhfg}
 Notation as in~\ref{adfdfhgf} and~\ref{asdgfssdtjhgkj}.   Let $n$ be a natural number.
Assume there is an element $v$  in $A$ with the following property. For every simplex 
 $\tau$ in $K$ such that  $\tau\cap (X\setminus A)\not=\emptyset$,  $\tau\cap (Y\setminus A)\not=\emptyset$, and 
 $|\tau\cap (K_0\setminus A)|\leq n+2$, the union  $\tau\cup\{v\}$ is also a simplex in $K$.
 Then, for every  simplex $\sigma$ in  $(\text{\rm sk}_{n+1}K)\setminus P$, the element $v$ is a central vertex (see~\ref{adghxdfgjhkut}) in $\text{\rm St}(\sigma,A)$. Furthermore the homotopy fibers of  $K_X\cup K_Y\subset K$ are $n$-connected.
\end{cor}
\begin{proof}
Let  $\sigma$ be a simplex in $ (\text{sk}_{n+1}K)\setminus P$. 
 If  $\mu$ belongs to $\text{\rm St}(\sigma,A)$ then,  by 
applying the assumption  of the corollary  to $\tau=\sigma\cup\mu$, we obtain that  $\sigma\cup\mu\cup\{v\}$ is a simplex in $K$ and hence  $\mu\cup\{v\}$ is  a simplex
in $\text{\rm St}(\sigma,A)$.  This means that $v$ is central in $\text{\rm St}(\sigma,A)$ (see~\ref{adghxdfgjhkut}).  Consequently  $\text{\rm St}(\sigma,A)$ is contractible (see~\ref{adghxdfgjhkut}) and the corollary follows from Proposition~\ref{sfsfdgfdhfg}.
\end{proof}

\section{Clique complexes}\label{dsfgdfhsfghj}
 Recall that  a simplicial complex $K$ is called 
 {\bf clique}  if it satisfies the following condition: a set  $\sigma$ of  size at least $ 2$ is a simplex in $K$ if and only if all the two element subsets of $\sigma$ are simplices in $K$.  Thus a clique complex is determined by its sets of vertices and edges. 

If $K$ is clique, then the   complexes $\text{\rm St}(\sigma, A)$ satisfy the following  properties:

\begin{prop}\label{asfgsdfhjghkhjgl}
Notation as in~\ref{adfdfhgf}. Assume $K$ is clique. Then:
\begin{enumerate}
\item  For all  $\sigma$ in $K$,
$\text{\rm St}(\sigma ,A)$ is  clique.
\item If   $\tau$ and $\sigma$ are  simplices in  $K$ such that  $\tau\cup\sigma$ is also a simplex in $K$, then
$\text{\rm St}(\tau\cup\sigma,A)=\text{\rm St}(\tau,A)\cap \text{\rm St}(\sigma,A)$. 
\item If $\sigma$ is a simplex in $K$ and $\sigma=\tau_1\cup\cdots\cup\tau_n$, then
$\text{\rm St}(\sigma,A)=\bigcap_{i=1}^{n}\text{\rm St}(\tau_i,A)$.
\item For every simplex $\sigma$ in $K$, $\text{\rm St}(\sigma,A)=\bigcap_{x\in \sigma}\text{\rm St}(\{x\},A)$.
\end{enumerate}
\end{prop}
\begin{proof}
Let $\mu$ be a subset of $A$  such that, for every 
two element subset $\tau$ of $\mu$, 
the set  $\tau\cup \sigma$ is a simplex in $K$, i.e., $\tau$ is in $\text{\rm St}(\sigma,A)$. Then, since $K$ is clique, $\mu\cup \sigma$ is also
a simplex in $K$. Consequently $\mu$ belongs to  $\text{\rm St}(\sigma, A)$ and   hence $\text{\rm St}(\sigma, A)$ is clique.
That proves (1).

To prove (2), first note that the inclusion $\text{\rm St}(\tau\cup\sigma,A)\subset \text{\rm St}(\tau,A)\cap \text{\rm St}(\sigma,A)$
holds even without the clique assumption. Let $\mu$ belong to both $\text{\rm St}(\tau,A)$ and $\text{\rm St}(\sigma,A)$.
This means that $\mu\cup\tau$ and $\mu\cup\sigma$ are simplices in $K$. Since every 2 element subset of 
$\mu\cup\tau\cup\sigma$ is a subset of either  $\mu\cup\tau$ or $\mu\cup\sigma$ or  $\tau\cup\sigma$,
by the assumption it is an edge in $K$. By the clique assumption,   $\mu\cup\tau\cup\sigma$ is then also a simplex in $K$ and consequently $\mu$ is in $\text{\rm St}(\tau\cup\sigma,A)$. This shows the other inclusion
$\text{\rm St}(\tau\cup\sigma,A)\supset \text{\rm St}(\tau,A)\cap \text{\rm St}(\sigma,A)$ proving (2).

Statements (3) and (4) follow from (2).
\end{proof}

Recall that an intersection of standard simplices is  again a standard simplex  (see~\ref{asfdfhd}).
This observation together with Propositions~\ref{sfsfdgfdhfg} and~\ref{asfgsdfhjghkhjgl} gives:
\begin{cor}\label{sfsgdhgjn}
Notation as in~\ref{adfdfhgf} and~\ref{asdgfssdtjhgkj}. Assume $K$ is clique and, 
for every edge $\tau$ in $K_1\setminus  P$, the complex $\text{\rm St}(\tau,A)$
is a standard simplex.  If, for all simplices $\sigma$ in $(\text{\rm sk}_{n+1}K)\setminus P$,  the complex
$\text{\rm St}(\sigma,A)$ is non-empty, then  the homotopy fibers of the 
 inclusion  $K_X\cup K_Y\hookrightarrow K$  are $n$-connected.
\end{cor}

Since clique complexes are determined by their edges, one can  wonder if, for such  complexes, 
the conclusions of Corollaries~\ref{asfdfgjfgjk,} and~\ref{afgsfdhfg} would still hold true 
 if  their   assumptions are verified  only for low dimensional simplices. 
 Here is an analogue   of Corollary~\ref{asfdfgjfgjk,}  for clique complexes.
 
 \begin{prop}\label{asfsdhowtreq}
 Notation as in~\ref{adfdfhgf}  and~\ref{asdgfssdtjhgkj}.  Let $n$ be a natural number.
 Assume  $K$ is clique and  that one of the following conditions is satisfied:
 \begin{enumerate}
 \item There is an $n$-connected simplicial complex  $L$ such that, for every  edge  $\tau$ in $K_1\setminus P$,  $\text{\rm St}(\tau,A)=L$.
 \item The complex $K_A$ is $n$-connected and, for every edge $\tau$ in   $K_1\setminus P$ 
 and every element  $v$ in $A$, the set $\tau\cup\{v\}$ is a simplex in $K$. 
   \end{enumerate}
 Then the homotopy fibers of the 
 inclusion  $K_X\cup K_Y\hookrightarrow K$  are $n$-connected.
 \end{prop}
 \begin{proof}
 The assumption (1) together with Proposition~\ref{asfgsdfhjghkhjgl}.(4) implies the assumption (1) of  
 Corollary~\ref{asfdfgjfgjk,}, proving the proposition in this case. 
 
  Let 
 $\tau$ be  an edge in  $K_1\setminus P$ and  $\mu$ be a simplex in $K_A$. Assume  (2).  This assumption  implies that any two element subset of $\tau\cup\mu$ is a simplex in $K$.  Since $K$ is clique, the set 
 $\tau\cup\mu$ is a simplex in $K$ and consequently $\mu$  is a simplex in $\text{\rm St}(\tau,A)$. 
 Thus for any $\tau$ in $K_1\setminus P$, there is an inclusion $K_A\subset \text{\rm St}(\tau,A)$,
 and hence $K_A= \text{\rm St}(\tau,A)$ for any such $\tau$. 
 The assumption (2) implies therefore the assumption (1) with $L=K_A$. 
  \end{proof}
 \begin{cor}\label{sasfdsdfhfdhjrtyj}
  Notation as in~\ref{adfdfhgf}  and~\ref{asdgfssdtjhgkj}.
   Assume  $K$ is clique and  that one of the following conditions is satisfied:
 \begin{enumerate}
 \item The set $A$ is non empty. Furthermore, for every edge $\tau$ in   $K_1\setminus P$ and every  subset $\mu$ in $A$ such that $|\mu|\leq 2$,  the union $\tau\cup\mu$ is a simplex in $K$. 
  \item  $A=\{v\}$ and, for every edge $\tau$ in    $K_1\setminus P$,   the set $\tau\cup\{v\}$ is also a simplex in $K$.
 \end{enumerate}
 Then the 
 inclusion  $K_X\cup K_Y\hookrightarrow K$ is a weak equivalence.
 \end{cor}
 \begin {proof}
 Assume (1). If $K_1\setminus P$ is empty, then so is $K\setminus P$, and hence $P=K$. In this case the corollary follows from Proposition~\ref{asfsdfg}. Assume $K_1\setminus P$ is non-empty.  Since any two element subset of $A$ is a simplex in $K$ and $K$ is clique, then all finite non-empty subsets of $A$ belong to  $K$ and hence $K_A=\Delta[A]$. In this case the assumption (1)
 is a particular case of the condition (2) in Proposition~\ref{asfsdhowtreq} for  all $n$ as $\Delta[A]$ is contractible. 
 
 Finally note that the assumption (2) is a particular case of (1). 
  \end{proof}
  The following is an analogue of Corollary~\ref{afgsfdhfg} which  is also referred to as ``one entry point".

 \begin{prop}\label{aSDFGDSFHFHGJKUI}
  Notation as in~\ref{adfdfhgf}  and~\ref{asdgfssdtjhgkj}.
 Assume  $K$ is clique and  that one of the following conditions is satisfied:
\begin{enumerate}
\item There is a vertex  $v$ in  $\bigcap_{\tau\in K_1\setminus P} \text{\rm St}(\tau,A)$ such that, for every edge $\tau$ in $K_1\setminus P$ and every vertex $w$ in  $\text{\rm St}(\tau,A)$, $\{v,w\}$ is a simplex in $K$.
\item There is a vertex  $v$ in  $\bigcap_{\tau\in K_1\setminus P} \text{\rm St}(\tau,A)$ such that, for every edge $\tau$ in $K_1\setminus P$, $v$ is a central vertex of $\text{\rm St}(\tau,A)$ (see~\ref{adghxdfgjhkut}). 
\item  There is an element $v$  in $A$ with the following property. For every 
simplex $\tau$ in $K$ such that  $|\tau\cap (X\setminus A)|=1$, $|\tau\cap (Y\setminus A)|=1$, and
$|\tau\cap A|\leq 1$,  the  union $\tau\cup\{v\}$ is also a simplex in $K$.
\end{enumerate}
 Then the 
 inclusion $K_X\cup K_Y\hookrightarrow K$ is a weak equivalence.
 \end{prop}
 \begin{proof}
 Assume (1). Let   $\sigma$  be a simplex  in $K\setminus P$. Choose a cover $\sigma=\tau_1\cup\cdots\cup \tau_n$ where $\tau_i$ is an edge in $K_1\setminus P$ for all $i$. Then according to Proposition~\ref{asfgsdfhjghkhjgl}, $\text{\rm St}(\sigma,A)=\bigcap_{i=1}^{n}\text{\rm St}(\tau_i,A)$. Let $w$ be a vertex in $\text{\rm St}(\sigma,A)$. Then it is also a vertex in $\text{\rm St}(\tau_i,A)$ for all $i$. By the assumption $\{v,w\}$ is then a simplex in $K$. 
 Thus all the 2 element subsets of $\sigma\cup\{v,w\}$ are simplices in $K$ and hence  $\{v,w\}$ is a simplex in
  $F(\sigma,A)$. 
 As this happens for all vertices $w$ in $\text{\rm St}(\sigma,A)$, since $\text{\rm St}(\sigma,A)$ is clique, for every simplex $\mu$ in $\text{\rm St}(\sigma,A)$, the set $\mu\cup\{v\}$ is  also a simplex in  $\text{\rm St}(\sigma,A)$.
 The vertex $v$ is therefore   central  in $\text{\rm St}(\sigma,A)$ and consequently $\text{\rm St}(\sigma,A)$ is contractible. The proposition under  assumption (1) follows then from 
 Corollary~\ref{sdfdfghsd}.(1).
 
 Condition (2) is a particular case of (1).
 
 Assume (3). Let  $\tau$ be an edge in $K_1\setminus P$.
 Condition (3) applied to the simplex $\tau$ gives that $\tau\cup\{v\}$ is a simplex in $K$ and hence
 $v$ is a vertex in $\text{\rm St}(\tau,A)$. 
Let $w$ be a vertex in $\text{\rm St}(\tau,A)$. Condition (3) applied to the simplex $\tau\cup\{w\}$   gives that $\{v,w\}\subset\tau\cup\{v,w\}$ are simplces in $K$. We can conclude  (3) implies (1).
 \end{proof}

We finish this section with   a statement 
  referred to as ``two entry points". This has been inspired by~\cite[Theorem 3]{NewHenry}, in which the gluing of two metric graphs along a path is considered. While in that case the two entry points are the endpoints of the path the graphs are glued along, in our framework they have to satisfy one of the listed properties. In both cases however these couple of points determine the weak equivalence stated.

  \begin{prop}\label{sdfhdghkyioui}
 Notation as in~\ref{adfdfhgf}  and~\ref{asdgfssdtjhgkj}.
Assume  $K$ is clique and   there are two elements $a_X$ and $a_Y$ in $A$ with the following properties:
\begin{itemize}
\item For every edge $\tau$ in  $K_1$ such that  $|\tau\cap A|=1$ and $|\tau\cap (X\setminus A)|=1$,
the set 
$\tau\cup\{a_X\}$ is a simplex in $K$.
\item For every edge $\tau$ in  $K_1$ such that  $|\tau\cap A|=1$ and $|\tau\cap (Y\setminus A)|=1$,
the set 
$\tau\cup\{a_Y\}$ is a simplex in $K$.
\item 
For every edge $\tau$   in $K_1\setminus P$, the set  $\tau\cup\{a_X,a_Y\}$ is a simplex in $K$.

Then, for every  $\sigma$ in $K\setminus P$, the set $\{a_X,a_Y\}$ is a central simplex (see~\ref{adghxdfgjhkut}) in $\text{\rm St}(\sigma,A)$, and 
the inclusion $K_X\cup K_Y\subset K$ is a weak equivalence.
\end{itemize}
\end{prop}
\begin{proof}
Let  $\sigma$ be a simplex in $K\setminus P$.  Any vertex $v$ in $\sigma$ is a vertex of an edge $\tau\subset \sigma$  that belongs to $K_1\setminus P$. According to the  assumption, the sets 
 $\{v,a_X,a_Y\}\subset \tau\cup\{a_X,a_Y\}$ are simplices in $K$. This, together  with the clique assumption on $K$, imply   $\sigma\cup\{a_X,a_Y\}$ is a simplex in $K$. Consequently $\{a_X,a_Y\}$ is a simplex in 
 $\text{\rm St}(\sigma,A)$. 
 
 Let $\mu$ be a simplex in $\text{\rm St}(\sigma,A)$. To prove the proposition, we need to show the set 
 $\mu\cup\{a_A,a_Y\}$ is a simplex in   $F(\sigma,A)$ or equivalently  $\sigma\cup \mu\cup\{a_X,a_Y\}$
 is a simplex in $K$. Let $x$ be an arbitrary element in $\sigma\cap X$, $y$ an   arbitrary element in $\sigma\cap Y$, and $v$ an arbitrary element in $\mu$.
The sets $\{x,v\}$, $\{y,v\}$, and $\{x,y\}$ are simplices in  $K$. Thus according to the assumptions so are
$\{x,v,a_X\}$,  $\{y,v,a_Y\}$, and $\{x,y,a_X,a_Y\}$. Consequently  the two element sets
$\{x,a_X\}$, $\{v,a_X\}$, $\{y,a_X\}$, $\{x,a_Y\}$, $\{v,a_Y\}$, $\{y,a_Y\}$, $\{a_X,a_Y\}$, $\{x,y\}$ are simplices in $K$.
Since  all the two element subsets of  $\sigma\cup\mu\cup\{a_X,a_Y\}$ are of such a form and  $K$ is clique,
 $\sigma\cup\mu\cup\{a_X,a_Y\}$ is a simplex in $K$.
\end{proof}

\section{Vietoris-Rips complexes for distances}
Let $Z$ be a subset of the universe $\mathcal{U}$ (see~\ref{aSFDHFN}). A function $d\colon Z\times  Z\to [0,\infty]$  is called a {\bf distance} if
it is symmetric $d(x,y)=d(y,x)$ and reflexive $d(x,x)=0$ for all  $x$ and  $y$ in $Z$.
A pair $(Z,d)$ is called a  distance space. A distance space $(Z,d)$ is sometimes denoted simply by $Z$,
if $d$ is understood from the context, or  by $d$, if $Z$ is understood from the context.

 Let  $(Z,d)$ be a  distance space. The {\bf diameter} of   a non empty and finite subset $\sigma \subset Z$ is by definition 
  $\text{diam}(\sigma):=\text{max}\{d(x,y)\ |\ x,y\in \sigma\}$. 
 
 A subset $X\subset Z$ together with the distance function given by the restriction of $d$ to $X$ is called
 a subspace of $(Z,d)$.

Let   $(Z,d)$ be a distance space and $r$ be in $[0,\infty)$.
By definition, the {\bf Vietoris-Rips} complex $\text{VR}_r(Z)$
consists of  these non-empty finite subsets $\sigma\subset Z$ for which    $\text{diam}(\sigma)\leq r$ (explicitly: $d(x,y)\leq r$ for all $x$ and $y$ in $\sigma$).  
Vietoris-Rips complexes are examples of clique  complexes (see Section~\ref{dsfgdfhsfghj}).

Let $X$ be a subspace of $(Z,d)$. Then the Vietoris-Rips complex $\text{VR}_r(X)$ coincides with the restriction
$\text{VR}_r(Z)_X$ (see~\ref{asfdfhd}).

Our starting assumption  in this section is:

\begin{point}[\bf Starting input II]\label{gsgfjyuolyupi}
{\em  $(Z,d)$  is a distance space, $X\cup Y=Z$ is a cover of $Z$, and   $A:=X\cap Y$.}
\end{point}

In the rest of this  section we are going to reformulate in terms of the distance $d$ on $Z$ some of the statements given  in the previous sections 
regarding the homotopy properties of the   inclusion $\text{\rm VR}_r(X)\cup \text{\rm VR}_r(Y) \hookrightarrow \text{\rm VR}_r(Z)$
for various $r$ in $[0,\infty)$.
Here is a direct restatement of Proposition~\ref{asfsdhowtreq}:

\begin{prop}\label{asewerwjj}
 Notation as in~\ref{gsgfjyuolyupi}.  Let $r$ be  an element in $[0,\infty)$
 and  $n$ be a natural number.
Assume that one of the following conditions is satisfied:
\begin{enumerate}
\item There is a subset $L\subset A$ such that $\text{\rm VR}_r(L)$ is $n$-connected and, 
for every $x$ in $X\setminus A$ and every $y$ in $Y\setminus A$ with  $d(x,y)\leq r$,
there is an equality $\{v\in A\ |\ d(x,v)\leq r\text{ and } d(y,v)\leq r\}=L$, in particular the set on the left does not depend on  $x$ and $y$.
\item The complex $\text{\rm VR}_r(A)$ is $n$-connected and, 
for  all $x$ in $X\setminus A$,  $y$ in $Y\setminus A$,  and  $v$ in $A$,
if $d(x,y)\leq r$, then both $d(x,v)\leq r$ and $d(v,y)\leq r$.
\end{enumerate}
Then the homotopy fibers of the  inclusion  $\text{\rm VR}_r(X)\cup \text{\rm VR}_r(Y)\subset
 \text{\rm VR}_r(Z)$
are $n$-connected.
\end{prop}

The assumption (2)  of Proposition~\ref{asewerwjj}  can be restated as: (connectivity condition) $\text{\rm VR}_r(A)$ is $n$-connected, and (intersection  condition)  if $\text{VR}_r(Z)_1\setminus P$ (see~\ref{asdgfssdtjhgkj}) is non-empty, then
\[A=\bigcap_{\sigma\in \text{VR}_r(Z)_1\setminus P} \text{\rm St}(\sigma,A)_0\] 
What if   the intersection above does not contain all   the points of $A$
(the intersection condition is not satisfied)?
For example consider $Z=\{x,a_1,a_2,a_3,a_4,y\}$ with the distance function depicted by the following diagram, where the dotted lines indicate distance 2 and the continuous lines indicate distance 1:
 
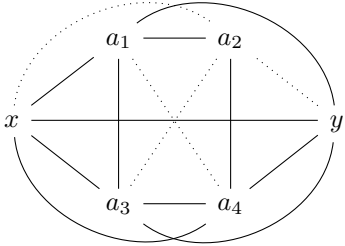
\[\begin{tikzcd}
  & a_1\ar[dash]{r} \ar[dash]{dd}\ar[dash, dotted]{ddr}& a_2\ar[dash]{dd}\ar[dash, dotted]{ddl}
 \\
 x \ar[dash]{ur} \ar[dash]{dr} \ar[dash, dotted, bend left=60pt]{urr}
 \ar[dash]{rrr}
 \ar[dash,  bend right=60pt]{drr}& & & y \ar[dash,  bend right=60pt]{ull}
  \ar[dash, dotted]{ul}\ar[dash]{dl} \ar[dash,   bend left=60pt]{dll}
 \\
 & a_3\ar[dash]{r} & a_4
  \end{tikzcd}\]
Let $X=\{x,a_1,a_2,a_3,a_4\}$ and $Y=\{a_1,a_2,a_3,a_4,y\}$. Choose $r=1$.
In this case  $\text{VR}_1(Z)_1\setminus P$ consists of only one  edge $\{x,y\}$ and
$\text{\rm St}(\{x,y\},A)=\Delta[\{a_1,a_3\}]\cup\Delta[\{a_3,a_4\}]$. Thus the condition (2) of Proposition~\ref{asewerwjj}  
is not satisfied.  However, since the complex $\text{\rm St}(\{x,y\},A)$ is contractible,  according to Corollary~\ref{sdfdfghsd}.(1),
the inclusion $\text{\rm VR}_1(X)\cup \text{\rm VR}_1(Y)\subset
 \text{\rm VR}_1(Z)$ is a weak equivalence.

 The assumption (1)   of Proposition~\ref{asewerwjj}   can be restated as: 
 (connectivity condition)  for every $\tau$ in $\text{VR}_r(Z)_1\setminus P$, the complex  $\text{\rm St}(\tau,A)$ is $n$-connected, and
 (independence condition) for all  pairs of edges $\tau_1$ and $\tau_2$ in $\text{VR}_r(Z)_1\setminus P$,
 there is an equality
 $\text{\rm St}(\tau_1,A)=\text{\rm St}(\tau_2,A)$.  What if the independence condition is not satisfied? For example  consider a distance space   $Z=\{x_1,x_2,a_1,a_2,a_3,a_4,y\}$  
 with the distance function depicted by the following diagram, where the dotted lines indicate distance 2 and the continuous lines indicate distance 1:
 
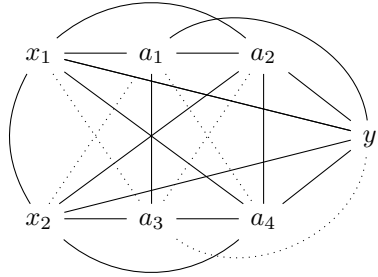
\[\begin{tikzcd}
 x_1 \ar[dash]{r} \ar[dash,  bend left=40pt]{rr}
 \ar[dash, bend right=30pt]{dd}
 \ar[dash,dotted]{ddr} \ar[dash]{rrdd}\ar[dash]{rrrd}&
  a_1\ar[dash]{r} \ar[dash]{dd}\ar[dash,dotted]{ddr}& 
  a_2
  \ar[dash]{dd}\ar[dash,dotted]{ddl}
 \\
 & & & y
 \ar[dash]{ul}
 \ar[dash]{dl}
 \ar[dash]{ulll}
 \ar[dash,  bend right=60pt]{ull}
 \ar[dash, dotted, bend left=60pt]{dll}
 \\
 x_2 \ar[dash,dotted]{uur}
 \ar[dash]{r}
  \ar[dash,   bend right=40pt]{rr}
  \ar[dash]{rruu}
  \ar[dash]{rrru}
 & a_3 
 \ar[dash]{r}
 & a_4
 \end{tikzcd}\]
 Let $X=\{x_1,x_2,a_1,a_2,a_3,a_4\}$ and $Y=\{a_1,a_2,a_3,a_4,y\}$. Choose $r=1$.
 In this case  $\text{VR}_1(Z)_1\setminus P$ consists of two edges $\{x_1,y\}$  and $\{x_2,y\}$.
 Note that $\text{\rm St}(\{x_1,y\},A)=\Delta[\{a_1,a_2\}]\cup \Delta[\{a_2,a_4\}]$ and $\text{\rm St}(\{x_2,y\},A)=\Delta[\{a_2,a_4\}]$. Thus the independence condition does not hold in this case. However,  since the obstruction complexes $\text{\rm St}(\{x_1,y\},A)$, $\text{\rm St}(\{x_2,y\},A)$  and $\text{\rm St}(\{x_1,x_2,y\},A)= \text{\rm St}(\{x_1,y\},A)\cap \text{\rm St}(\{x_2,y\},A)=\Delta[\{a_2,a_4\}]$ are contractible, 
 the inclusion $\text{\rm VR}_1(X)\cup \text{\rm VR}_1(Y)\subset
 \text{\rm VR}_1(Z)$ is a weak equivalence by Corollary~\ref{sdfdfghsd}.

 Consider  a relaxation of the intersection condition in the assumption (2) of Proposition~\ref{asewerwjj}.
\begin{point}[\bf  Assumption I]\label{adgfsdfhsdg} {\em Notation as in~\ref{gsgfjyuolyupi}. 
Let $r$ be  an element in $[0,\infty)$.
 There exists an element $v$ in $A$ satisfying the following property. 
 For all  $x$ in $X\setminus A$ and $y$ in $Y\setminus A$, if   $d(x,y)\leq r$, then  
  $d(x,v)\leq r$ and $d(y,v)\leq r$.}
 \end{point}
 
 \begin{point}\label{basfasfgdqweter}
 The  assumption~\ref{adgfsdfhsdg} is   equivalent to non-emptiness of the following intersection, where
 $n\geq 0$ and  the first equality is a consequence of Vietoris-Rips complexes being clique (see Proposition~\ref{asfgsdfhjghkhjgl}):
\[\bigcap_{\sigma\in (\text{sk}_{n+1}\text{VR}_r(Z))\setminus P} \text{\rm St}(\sigma,A)=\bigcap_{\sigma\in  \text{VR}_r(Z)_1\setminus P} \text{\rm St}(\sigma,A)\not=\emptyset\]
According to Proposition~\ref{adgsfgjkk}, Assumption~\ref{adgfsdfhsdg} implies connectedness of the 
 homotopy fibers of $\text{\rm VR}_r(X)\cup \text{\rm VR}_r(Y)\subset
 \text{\rm VR}_r(Z)$.  This  assumption however does not imply $n$-connectedness of these homotopy fibers. For example  consider $Z=\{x,a,b,y\}$ with the distance function depicted by the following diagram  where the dotted line indicates distance 2 and the continuous lines indicate distance 1:
 
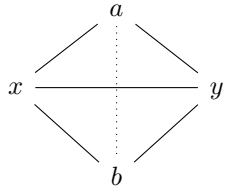
\[\begin{tikzcd}
 & a\ar[dash]{dl}\ar[dash]{dr}\ar[dash,dotted]{dd}\\
 x \ar[dash]{dr}\ar[dash]{rr}& & y\ar[dash]{dl}\\
 & b
 \end{tikzcd}\]
 Let $X=\{x,a,b\}$ and $Y=\{a,b,y\}$. Then $\text{\rm VR}_1(Z)$ is contractible but $\text{\rm VR}_1(X)\cup \text{\rm VR}_1(Y)$ 
 has the homotopy type of a circle.  Thus 
 in this case the homotopy fiber of the inclusion $\text{\rm VR}_r(X)\cup \text{\rm VR}_r(Y)\subset
 \text{\rm VR}_r(Z)$ is not $1$-connected.
 Note further that  the complex  $\text{St}(\{x,y\},A)$ consists of two vertices $a$ and $b$ with no edges. 
 \end{point}

To assure   $\text{\rm VR}_r(X)\cup \text{\rm VR}_r(Y) \hookrightarrow \text{\rm VR}_r(Z)$ is a weak equivalence  assumption~\ref{adgfsdfhsdg} is not enough and we need additional 
requirements.  For example the following is an analogue of Corollary~\ref{sasfdsdfhfdhjrtyj}.

\begin{prop}\label{afadfhhgj}
 Notation as in~\ref{gsgfjyuolyupi}. Assume~\ref{adgfsdfhsdg}. Let $v$ be an element in $A$ given by this assumption. In addition assume that one of the following conditions is satisfied:
\begin{enumerate}
\item For every $x$ in $X\setminus A$ and $y$ in $Y\setminus A$ such that $d(x,y)\leq r$, if $w$ in $A$ satisfies  $d(w,x)\leq r$ and $d(w,y)\leq r$, then $d(v,w)\leq r$.
\item {\em $\text{\rm diam}(A)\leq r$.}
\item {\em $A=\{v\}$.}
\end{enumerate}
Then the     inclusion $\text{\rm VR}_r(X)\cup \text{\rm VR}_r(Y) \hookrightarrow \text{\rm VR}_r(Z)$ is a weak equivalence. 
\end{prop}
\begin{proof}
If Assumption~\ref{adgfsdfhsdg} and the condition 1 hold, then so does  the assumption 1 of Proposition~\ref{aSDFGDSFHFHGJKUI}.
Furthermore  the condition 3 implies 2 and the condition  2 implies 1. Thus this proposition is  a consequence of Proposition~\ref{adfdfhgf}.
\end{proof}



The intersection condition of the assumption (2) in Proposition~\ref{asewerwjj}
requires a choice of a parameter $r$.  The following is its  universal version  where no parameter is required:

\begin{point}[\bf  Assumption II]\label{afdsghdfhkte} {\em Notation as in~\ref{gsgfjyuolyupi}. 
 The set $A$ is non empty and  for every    $x$   in $X\setminus A$,  $y$    in $Y\setminus A$,  and $v$   in $A$, the following inequalities hold $d(x,y)\ge d(x,v)$ and $d(x,y)\ge d(y,v)$.}
\end{point}

Assumption~\ref{afdsghdfhkte} has an intuitive interpretation in terms of angles when 
$Z$ is a subspace of the Euclidean space. In such a setting this condition means that 
every triangle  $xvy$ with vertices $x$ in $X\setminus A$, $y$ in $Y\setminus A$ and $v$ in $A$, 
the angle  at $v$ must be at least $60^{\circ}$. We therefore refer to this assumption as the  $60^{\circ}$
angle condition.


\begin{prop}\label{asfgsdfgjtoyui}
 Notation as in~\ref{gsgfjyuolyupi}. Assume~\ref{afdsghdfhkte}. Assume in addition that,
  for every    $x$   in $X\setminus A$ and  $y$    in $Y\setminus A$, the following inequality holds 
$d(x, y)\geq \text{\rm diam}(A)$. Then $\text{\rm VR}_r(X)\cup \text{\rm VR}_r(Y) \hookrightarrow \text{\rm VR}_r(Z)$ is a weak equivalence for all $r$ in $[0,\infty)$.
\end{prop}
\begin{proof}
We already know that the proposition holds if  $r\geq \text{diam}(A)$. Assume  $r< \text{diam}(A)$.
We claim  that in this case $\text{\rm VR}_r(Z)=P$ (see~\ref{asdgfssdtjhgkj}). If not, there are $x$ in $X\setminus A$ and $y$ in $Y\setminus A$ such that $d(x,y)\leq r$.  The  assumption   would then lead to the following contradictory inequalities
$r\geq d(x,y)\geq \text{diam}(A)>r$.
Thus in this case   $\text{\rm VR}_r(Z)=P$ and the proposition follows from Proposition~\ref{asfsdfg}.
\end{proof}

\section{Metric gluings}
A distance $d$ on $Z$ is called a  {\bf pseudometric} if it satisfies the triangle inequality: $d(x,z)\leq d(x,y)+d(y,z)$
 for all $x,y,z$ in $Z$.


\begin{point} \label{sdasfADSFDHGSJ}
Notation as in~\ref{gsgfjyuolyupi}.  Assume that the distance $d$ on $Z$ is a pseudometric.
 Let  $x$ be  in $X\setminus A$ and $y$ be in  $Y\setminus A$.  For all $a$ in $A$, by the triangular inequality, $d(x,y)\leq d(x,a)+d(a,y)$, and hence:
\[d(x,y)\leq \text{inf}\{d(x,a)+d(a,y)\ |\ a\in A\}\]
The pseudometric space $(Z,d)$ is called {\bf metric gluing} if the above inequality is an equality for
all $x$   in $X\setminus A$ and $y$  in  $Y\setminus A$. 

If $A$ is finite, then the pseudometric $(Z,d)$ is a metric gluing if and only if,  for every $x$ in $X\setminus A$ and $y$  in  $Y\setminus A$, there is
$a$ in $A$ such that $d(x,y)=d(x,a)+d(a,y)$.

If $d_X$ is a pseudometric on $X$ and $d_Y$ is a pseudometric on $Y$ such that $d_X(a,b)=d_Y(a,b)$ for all $a$ and $b$ in $A$, then the following function defines a pseudometric on $Z$ which is a metric gluing:
\[d_Z(z,z')=
\begin{cases} 
      d_X(z,z') & \text{ if } z, z' \in X  \\
      d_Y(z,z') & \text{ if } z, z' \in Y  \\
      \inf \{d(z,a)+d(z',a)\ |\  a\in A\} & \text{ if } z\in X\setminus A \text{ and } z'\in Y\setminus A
   \end{cases}
\]
\end{point}

If    $(Z,d)$ is a  metric gluing and $A$ is finite, then, for any edge $\sigma=\{x,y\}$ in $\text{VR}_r(Z)\setminus P$, there is $a$ in $A$ such that $r\geq d(x,y)=d(x,a)+d(a,y)$. Thus in this case  the obstruction complex $\text{\rm St}(\tau,A)$ is non-empty as it contains the vertex $a$.  To assure  contractibility of $\text{\rm St}(\tau,A)$ we need additional assumptions, for example:

\begin{point}[\bf  Simplex assumption] \label{SADSDGHDSG}
Notation as in~\ref{gsgfjyuolyupi}   and~\ref{asdgfssdtjhgkj}.
Let $r$ be in $[0,\infty)$. For any vertex $v$ in an edge $\sigma$ in $\text{VR}_r(Z)\setminus P$,
if $a$ and $b$ are elements in $A$ such that $d(a,v)\leq r$ and $d(v,b)\leq r$, then $d(a,b)\leq r$.
\end{point}
The simplex assumption  can be reformulated as follows: for any vertex $v$ in a simplex  $\sigma$ in $\text{VR}_r(Z)\setminus P$, the complex $\text{\rm St}(v,A)$ is a standard simplex (see~\ref{aSFDHFN}). Since the intersection of standard simplices is again a standard simplex, under  Assumption~\ref{SADSDGHDSG}, 
 an obstruction complex $\text{\rm St}(\sigma,A)$, for an arbitrary simplex $\sigma$ in  $\text{VR}_r(Z)\setminus P$, is contractible if and only if it is  non empty.   This, together with the discussion at the end of~\ref{sdasfADSFDHGSJ}  and Corollary~\ref{sfsgdhgjn} gives:
 
\begin{prop}\label{sdasDSFAGHF}
Notation as in~\ref{gsgfjyuolyupi}    and~\ref{asdgfssdtjhgkj}.  Let  $r$  be in $[0,\infty)$. Assume $A$ is finite and $(Z,d)$ is a   metric gluing that satisfies the simplex assumption~\ref{SADSDGHDSG}. Then, for any edge $\sigma$ in  $\text{VR}_r(Z)\setminus P$, the obstruction complex $\text{\rm St}(\sigma, A)$ is contractible. The homotopy fibers of 
$\text{\rm VR}_r(X)\cup \text{\rm VR}_r(Y) \hookrightarrow \text{\rm VR}_r(Z)$  are connected and this map
induces an isomorphism on $\pi_0$ and a surjection on $\pi_1$.
\end{prop}

The assumptions of Proposition~\ref{sdasDSFAGHF}  are not enough to guarantee the non-emptiness of the obstruction complexes 
for simplices   in $\text{VR}_r(Z)\setminus P$  of dimension 2 and higher. 
 For example consider $Z=\{x_1,x_2,a_1,a_1,y\}$ with the distance function depicted by the following diagram, where the dotted lines indicate distance 4, the dashed lines indicate distance 3, the squiggly lines indicate distance 2 and the continuous lines indicate distance $1$:
 
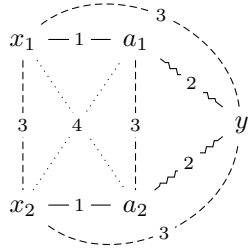
\[\begin{tikzcd}[labels=description]
 x_1\ar[dash]{r}{1}\ar[dash, dotted]{ddr}{1}\ar[dash,dashed]{dd}{3}\ar[dash,dashed,  bend left=60pt]{drr}{3}   & a_1\ar[dash, dashed]{dd}{3}
 \arrow[dash,squiggly]{dr}{2}\\
 & & y\\
 x_2\ar[dash, dotted]{ruu}{4} \ar[dash]{r}{1}\ar[dash, dashed,  bend right=60pt]{urr}{3} & a_2\ar[dash,squiggly]{ur}{2}\\
  \end{tikzcd}\]
  Let $r=3$, $X=\{x_1,x_2,a_1,a_2\}$, $Y=\{y,a_1,a_2\}$, and $A=\{a_1,a_2\}$.  Then  $Z$ is a metric gluing. Note that  $\text{\rm St}(y,A)=\Delta[\{a_1,a_2\}]$, $\text{\rm St}(x_1,A)=\text{\rm St}(\{x_1,y\},A)=\Delta[\{a_1\}]$,
  $\text{\rm St}(x_2,A)=\text{\rm St}(\{x_2,y\},A)=\Delta[\{a_2\}]$, and    $\text{\rm St}(\{x_1,x_2,y\},A)$ is empty.  Furthermore
  $A$ and $Y$ have diameter not exceeding  $3$. Consequently  
$\text{VR}_3(Y)$  and $\text{VR}_3(A)$ are contractible.  The complex $\text{VR}_3(X)$ has the homotopy type of the circle  $S^1$ and so does   $ \text{VR}_3(X)\cup \text{VR}_3(Y)$.  The entire complex $\text{VR}_3(Z)$ is however contractible. The inclusion $ \text{VR}_3(X)\cup \text{VR}_3(Y)\subset \text{VR}_3(Z)$ induces therefore a surjection on $\pi_1$ but  not an isomorphism. This example should be compared with
Proposition~\ref{afadfhhgj} under the condition 2.

To assure isomorphism on $\pi_1$, the simplex assumption~\ref{SADSDGHDSG} should be strengthened.

\begin{point}[\bf Strong  simplex assumption]\label{sasdfgsdfhgjd}
Notation as in~\ref{gsgfjyuolyupi}    and~\ref{asdgfssdtjhgkj}.
Let $r$ be in $[0,\infty)$. For any vertex $v$ in an edge $\sigma$ in $\text{VR}_r(Z)\setminus P$,
if $a$ and $b$ are elements in $A$ such that $d(a,v)\leq r$ and $d(v,b)\leq r$, then $2d(a,b)\leq d(a,v)+d(v,b)$.
\end{point}

Note that the strong  simplex assumption~\ref{sasdfgsdfhgjd} implies the  simplex assumption~\ref{SADSDGHDSG}. 
\begin{thm}\label{easdfgsdghj}
Notation as in~\ref{gsgfjyuolyupi}    and~\ref{asdgfssdtjhgkj}.  Let  $r$  be in $[0,\infty)$. Assume $A$ is finite and $(Z,d)$ is a   metric gluing that satisfies the strong simplex assumption~\ref{sasdfgsdfhgjd}.  Then,  for any simplex $\sigma$  in  $\text{\rm VR}_r(Z)\setminus P$
such that  either $|\sigma\cap X|=1$ or $|\sigma\cap Y|=1$, the obstruction complex  $\text{\rm St}(\sigma,A)$ is   contractible. The homotopy fibers of the inclusion $\text{\rm VR}_r(X)\cup \text{\rm VR}_r(Y) \hookrightarrow \text{\rm VR}_r(Z)$  are simply connected and this map
induces an isomorphism on $\pi_0$ and $\pi_1$  and a surjection on $\pi_2$.
\end{thm}
\begin{proof}
Since the  strong simplex assumption~\ref{sasdfgsdfhgjd} is satisfied, then so is the simplex assumption~\ref{SADSDGHDSG}  and consequently any obstruction complex  $\text{\rm St}(\sigma,A)$ is a simplex. Thus  $\text{\rm St}(\sigma,A)$  is contractible if and only if it is non empty.

We are going to show by induction on the dimension of a simplex   a more general   statement: 
 \smallskip

\noindent
{\em 
Under the assumption of Theorem~\ref{easdfgsdghj}, for every simplex  $\sigma$ in $\text{\rm VR}_r(Z)\setminus P$  for which  
  $\sigma \cap X=\{x_1,\dots,x_n\}$ and  $\sigma\cap Y=\{y\}$, if $(a_1,\ldots,a_n)$ is a sequence in $A$ such that
  $d(x_i,y)=d(x_i,a_i)+d(a_i,y)$ for every $i$, then there is $l$  for which $a_l$ is in $\text{\rm St}(\sigma,A)$ ($d(x_i,a_l)\leq r$ for all $i$).}
 \smallskip

If $\sigma=\{x,y\}$ is such an edge, then the statement is clear.

Let $n>1$ and assume that the statement is true for all relevant simplices  of dimension smaller than $n$.
Let $\sigma$ be in $\text{\rm VR}_r(Z)\setminus P$ be such   that 
  $\sigma \cap X=\{x_1,\dots,x_n\}$ and  $\sigma\cap Y=\{y\}$. 
  Choose a sequence $(a_1,\ldots,a_n)$ in $A$ such that
  $d(x_i,y)=d(x_i,a_i)+d(y,a_i)$ for every $i$.  
 
  By the inductive assumption, for every $j=1,\ldots, n$, the statement is true for  $\tau_j=\sigma_j\setminus\{x_j\}$ and the sequence $(a_1,\ldots,\widehat{a_j},\ldots,a_n)$
  obtained from  $(a_1,\ldots,a_n)$ 
  by removing its  $j$-th element.  Thus for every $j=1,\ldots, n$, there is $a_{s(j)}$  such that $s(j)\not= j$ and $d(x_i,a_{s(j)})\leq r$ for all
  $i\not=j$. If, for some $j$, $d(x_j,a_{s(j)})\leq r$, then $a_{s(j)}$ would be a vertex in   $\text{\rm St}(\sigma,A)$, proving the statement.
  Assume $d(x_j,a_{s(j)})> r$ for all $j$.  If $j\not=j'$, then $d(x_j, a_{s(j')})\leq r$ and  $d(x_j, a_{s(j)})> r$, and hence 
   $a_{s(j)}\not=a_{s(j')}$.  It follows that $s$ is a permutation of the set $\{1,\ldots,n\}$. This together with the strong simplex assumption  leads to a contradictory inequality:
   \begin{align*}
nr & < \sum _{i=1}^n d(x_i,a_{s(i)}) \le \sum _{i=1}^n \left(d(x_i,a_{i})+d(a_i,a_{s(i)})\right)\leq\\
& \le \sum _{i=1}^n \left(d(x_i,a_i)+\frac{1}{2}d(a_i,y)+\frac{1}{2}d(y,a_{s(i)})\right)\le \\
&\le \sum _{i=1}^n  \left(d(x_i,a_i)+d(y,a_i)\right) \le nr
\end{align*}

Note that, for any simplex $\sigma$  in  $\text{\rm sk}_{2}\text{VR}(Z)\setminus P$, either $|\sigma\cap X|=1$ or $|\sigma\cap Y|=1$.
Thus, for any such simplex, the obstruction complex $\text{St}(\sigma,A)$ is contractible. We can  then use Proposition~\ref{sfsfdgfdhfg}
to conclude that the homotopy fibers of the inclusion $\text{\rm VR}_r(X)\cup \text{\rm VR}_r(Y) \hookrightarrow \text{\rm VR}_r(Z)$
are  simply connected. 
\end{proof}

The conclusion of Theorem~\ref{easdfgsdghj} is sharp.
Its  assumptions  are not enough to assure  that  the homotopy fibers of the map
$\text{\rm VR}_r(X)\cup \text{\rm VR}_r(Y) \hookrightarrow \text{\rm VR}_r(Z)$  are $2$ connected.
We finish this section with an example illustrating this  fact.

\begin{point}
Let    $Z=\{x_1,x_2,a_{11},a_{12},a_{21}, a_{22},y_1,y_2\}$,  
$X=\{x_1,x_2,a_{11},\allowbreak a_{12},a_{21},a_{22}\}$ and $Y=\{y_1,y_2,a_{11},a_{12},a_{21},a_{22}\}$.  
Consider the distance  function $d$ on $Z$ described by the following table:
\[\begin{array}{c|c|c|c|c|c|c|c}
 &x_2& a_{11}& a_{12} &a_{21}& a_{22}&y_1&y_2\\ \hline
x_1  & 6 & 3 & 5 &  7 & 9 &8 & 8 \\ \hline
x_2   & 0 & 9 & 7 & 5 & 3 & 8 & 8 \\ \hline
a_{11}  & & 0 &4 & 4 & 6 & 5 & 7 \\ \hline
a_{12}  & & & 0 & 6 & 4 & 9 & 3 \\ \hline
a_{21}  & & & & 0 & 4 & 3 & 9 \\ \hline
a_{22}  & & & & & 0 & 7 & 5 \\ \hline
y_1      & & & & & & 0 & 6\\
\end{array}\]
The distance $d$ satisfies the triangular inequality and hence $(Z,d)$ is a metric space. Furthermore
$d(x_i,y_j)=d(x_i,a_{ij})+d(a_{ij},y_j)$ for any $i$ and $j$. Thus $(Z,d)$ is a metric gluing  of $X$ and $Y$. 
The metric space $(Z,d)$ can be represented by the following diagram, where the 
continuous  lines or no line indicate distance  $8$ or smaller and   the dotted lines indicate distance $9$:
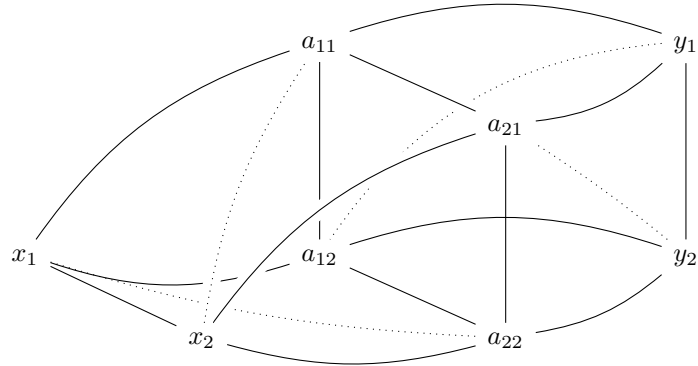
\[\begin{tikzcd}[labels=description]
 & &  & a_{11} \ar[dash]{ddd}  & & & & y_1\ar[dash, bend right=18pt]{llll}\ar[dash]{ddd}
 \ar[dash, bend left=18pt]{dll}
\ar[dash,dotted, bend right=25pt]{llllddd}
 \\
  & & & & & a_{21}\ar[dash]{ddd}
  \ar[dash,crossing over,from=ull]
 \\
  & &\\
x_1\ar[dash,bend left=18pt]{uuurrr}\ar[dash,dotted, bend right=7pt]{drrrrr}
\ar[dash,bend right=18pt]{rrr}\ar[dash]{drr}
& & & a_{12}\ar[dash]{drr} & & &  & y_2\ar[dash, bend right=18pt]{llll}\ar[dash, bend left=15pt]{dll}
\ar[dash,dotted, bend right=3pt]{uull}
\\
 & &  x_2\ar[dash,bend left=18pt,crossing over]{uuurrr}\ar[dash,bend right=15pt]{rrr}\ar[dash,dotted, bend left=12pt,crossing over]{uuuur}
& & & a_{22}
\end{tikzcd}
\]                                                                                              
By direct calculation one checks that, for $r=8$ and $Z=X\cup Y$, the metric space $(Z,d)$ satisfies the strong simplex assumption~\ref{sasdfgsdfhgjd}. However,  the complex $\text{\rm VR}_8(X)\cup \text{\rm VR}_8(Y)$ is contractible and $ \text{\rm VR}_8(Z)$ is weakly equivaent to $S^3$ ($3$-dimensional sphere).The homotopy fiber of
$\text{\rm VR}_8(X)\cup \text{\rm VR}_8(Y)\subset \text{\rm VR}_8(Z)$ is therefore  weakly equivalent to the loops space $\Omega S^3$ and hence is not $2$ connected.

Here are steps that one might use to see that $ \text{\rm VR}_8(Z)$ is weakly equivalent to $S^3$.
Consider the simplex $\{x_1,x_2\}$ in $ \text{\rm VR}_8(Z)$. According to Corollary~\ref{asfgdsgfjghkkl}, there is a homotopy  cofiber  sequence:
\[\Sigma \text{St}\left(\{x_1,x_2\}, \text{VR}_8(Z\setminus\{x_1,x_2\}\right)\to \text{VR}_8(Z\setminus\{x_1\})\cup  \text{VR}_8(Z\setminus\{x_2\})\to \text{VR}_8(Z) \]
The complexes in this sequence have the following homotopy types:
\begin{itemize}
\item $\text{St}\left(\{x_1,x_2\}, \text{VR}_8(Z\setminus\{x_1,x_2\}\right)$ is weakly equivalent to the circle $S^1$;
\item $\text{VR}_8(Z\setminus\{x_1\})$, $ \text{VR}_8(Z\setminus\{x_2\})$, $\text{VR}_8(Z\setminus\{x_1.x_2\})$
are contractible;
\item the above implies that $\text{VR}_8(Z\setminus\{x_1\})\cup  \text{VR}_8(Z\setminus\{x_2\})$  is also contractible;
\item we can then use the cofiber sequence above to conclude  $\text{VR}_8(Z)$ is weakly  equivalent to  $\Sigma^2 S^1\simeq S^3$ as claimed.
\end{itemize}

\end{point}

\section{Vietoris-Rips of 9 points on a circle}
In~\cite{MR3096593,MR3673078,MR3530967} a lot of techniques were introduced aiming at  describing homotopy types of certain Vietoris-Rips complexes, particularly for metric graphs built from points on a circle. 
In this section we showcase how our techniques can be used to describe the homotopy type of  one of such examples.
Consider a metric space given by $9$ points, $Z=\{z_i\}_i$ with  the following distances between them:
 \[\begin{array}{c|c|c|c|c|c|c|c|c}
 & z_2& z_3 &z_4& z_5&z_6&z_7&z_8&z_9\\\hline
z_1  & 1 & 2 &  3 & 4 &4 &3&2&1 \\ \hline
z_2   &0&1&2&3&4&4&3&2 \\ \hline
z_3  &&0&1&2&3&4&4&3 \\ \hline
z_4  &&&0&1&2&3&4&4\\ \hline
z_5  &&&&0&1&2&3&4 \\ \hline
z_6  &&&&&0&1&2&3 \\ \hline
z_7  &&&&&&0&1&2\\ \hline
z_8  &&&&&&&0&1\\ 
 
\end{array}\]
This metric space can be visualised as a metric graph consisting of   $9$ points on a circle, where distances between adjacent points are set to be  $1$:
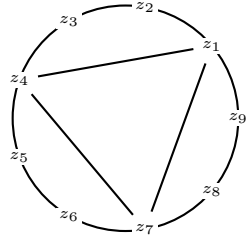
\[
\begin{tikzpicture}[scale=0.6]
\draw [thick,fill=white] (0,0) circle (2.5cm);
\foreach \a in {1,2,...,9}{
    \node [font=\tiny, fill=white,inner sep=1pt] at (\a*360/9: 2.5cm){$\textstyle{z}_{\a}$};}
\draw  [thick, shorten <= 0.25cm, shorten >= 0.25cm](1*360/9: 2.5cm) - -  (4*360/9: 2.5cm);
\draw [ thick, shorten <= 0.25cm, shorten >= 0.25cm] (1*360/9: 2.5cm) - -  (7*360/9: 2.5cm);
\draw  [thick, shorten <= 0.25cm, shorten >= 0.25cm](7*360/9: 2.5cm) - -  (4*360/9: 2.5cm);
\end{tikzpicture}
\]
We are going to illustrate how to  use our techniques to prove  that $\text{VR}_3(Z)$ is weakly equivalent to the wedge $S^2\vee S^2$ of two $2$-dimensional spheres, a result already present in the mentioned work of Adamaszek et al.
Set $X:=\{ z_1,z_2,z_4,z_5,z_7,z_8\}$ and $Y:=\{z_1,z_3,z_4,z_6,z_7,z_9\}$. Note that  $X\cup Y=Z$ and
$A:=X\cap Y=\{z_1,z_4,z_7\}$.

Note that $\text{VR}_3(X\cap Y)$ is contractible, and thus  $\text{VR}_3(X)\cup \text{VR}_3(Y)$ is homotopy equivalent to the  wedge $\text{VR}_3(X)\lor \text{VR}_3(Y)$. Furthermore we claim that all the obstruction complexes $\text{St}(\sigma,A)$
are contractible for all simplices $\sigma$ in $\text{VR}_3(Z)$ such that $\sigma\cap X\not=\emptyset$, $\sigma\cap Y\not=\emptyset$,
and $\sigma\cap X\cap Y=\emptyset$.
For example if  $\sigma=\{x_2,x_3\}$ then  $\text{St}(\sigma,A)=\Delta[x_1,x_7]$  which is contractible.
According to Corollary ~\ref{sdfdfghsd}, the inclusion  $VR_3(X)\cup VR_3(Y)\subset VR_3(Z)$ is therefore a weak equivalence
and consequently  $VR_3(Z)$ has the homotopy type of the wedge $\text{VR}_3(X)\lor \text{VR}_3(Y)$. 
The metric spaces $X$ and $Y$ are isometric, and hence the corresponding Vietoris-Rips complexes are isomorphic. It remains to show that $VR_3(X)$ has the homotopy type of $S^2$.  Consider  $X'=\{z_1,z_5\}$ and $X''=\{z_2,z_4,z_7,z_8\}$.
Note that $X=X'\coprod X''$, $\text{VR}_3(X')$ has the homotopy type of $S^0$ and $\text{VR}_3(X'')$ has the homotopy type of $S^1$. Finally note that, for all simplices $\sigma$ in  $\text{VR}_3(X')$ and   $\mu$ in  $\text{VR}_3(X'')$,
the union $\sigma\cup \mu$ is a simplex in  $VR_3(X)$. This implies that $VR_3(X)$ is  the join of $\text{VR}_3(X')$ and $\text{VR}_3(X'')$
(see Paragraph~\ref{afsaasfagdf}) and hence it is weakly equivalent to $\Sigma (S^0\wedge S^1)\simeq S^2$.

\paragraph{\em Acknowledgments.}    A.\ Jin and F.\ Tombari 
were supported by the Wallenberg AI, Autonomous System and Software Program (WASP) funded by Knut and Alice Wallenberg Foundation.
W.\ Chach\'olski was partially supported by VR and WASP.
M.\ Scolamiero was partially supported by Brummer \& Partners MathDataLab  and WASP.
\bibliographystyle{spmpsci}      
\bibliography{references.bib}   

%
%

\end{document}